\documentclass[11pt]{amsart}

\usepackage{times}
\usepackage{amssymb}
\usepackage{graphicx,xspace}
\usepackage{epsfig}
\usepackage{enumitem}
\usepackage{xcolor}
\usepackage{tikz}
\usetikzlibrary{decorations.markings}
\usepackage{gensymb}
\usepackage[T1]{fontenc}
\usepackage[utf8]{inputenc}
\usepackage{bm}
\usepackage{hyperref}
\usepackage[capitalise]{cleveref}
\usepackage{todonotes}

\newtheorem{theorem}{Theorem}

\theoremstyle{plain}

\newtheorem{lemma}{Lemma}[section]
\newtheorem{corollary}[lemma]{Corollary}
\newtheorem{proposition}[lemma]{Proposition}

\theoremstyle{remark}
\newtheorem{remark}[lemma]{Remark}
\newtheorem{example}[lemma]{Example}
\newtheorem{special_case}[lemma]{Special case}
\newtheorem{definition}[lemma]{Definition}

\newcommand{\s}[1]{\mathbf{#1}}

\newcommand{\C}{\mathbb{C}}
\newcommand{\N}{\mathbb{N}}
\newcommand{\M}{\mathbf{M}}

\newcommand{\F}{\mathbf{F}}
\renewcommand{\L}{\mathbf{L}}

\newcommand{\ov}[1]{\overline{#1}}
\renewcommand{\u}[1]{{#1}_\star}

\newcommand{\Z}{\mathbb{Z}}

\newcommand{\II}{\mathbf{I}}
\newcommand{\JJ}{\mathbf{J}}
\newcommand{\KK}{\mathbf{K}}
\newcommand{\kk}{\mathbf{k}}
\newcommand{\FF}{\mathbf{F}}
\newcommand{\NN}{\mathbf{N}}

\newcommand{\PP}{\mathbf{P}}

\DeclareMathOperator{\CIE}{CIE}

\DeclareMathOperator{\Ch}{Ch}

\DeclareMathOperator{\std}{std}

\usepackage{mathrsfs}
\newcommand{\G}{\mathscr{G}}
\newcommand{\BG}{\mathscr{G}_b}
\def\c{c}
\def\ideal{\mathscr{C}}
\def\Ker{\mathscr{K}}
\def\Bideal{\mathscr{C}_b}
\def\BKer{\mathscr{K}_b}
\newcommand{\FFF}{\mathscr{F}}
\def\GammaNC{\Gamma^{\text{nc}}}
\def\BGammaNC{\Gamma^{\text{nc}}_b}
\def\BGamma{\Gamma_b}
\def\oG{\overline{G}}
\def\QSym{{\it QSym}}

\def\WQSym{{\bf WQSym}}
\def\bv{{\bf v}}
\def\bw{{\bf w}}
\def\eval{{\rm eval}}

\def\gf#1#2{\genfrac{}{}{0pt}{}{#1}{#2}}
\renewcommand{\phi}{\varphi}
\def\comm{\phi_c}
\def\ex{\text{ex}}
\def\MP{\text{MP}}
\def\NNN{\mathscr{N}}

\author[V.~Féray]{Valentin Féray}

\address{Institut für Mathematik, Universität Zürich, Winterthurerstrasse 190, 8057 Zürich, Switzerland}
\email{valentin.feray@math.uzh.ch}

\thanks{
VF has been partially supported by ANR project PSYCO ANR-11-JS02-001 and by SNSF grant ``Dual combinatorics of Jack polynomials''.}



\title
{Cyclic inclusion-exclusion}

\begin{document}

\maketitle

\begin{abstract}
    Following the lead of Stanley and Gessel, we consider a morphism
    which associates to an acyclic directed graph (or a poset) a quasi-symmetric function.
    The latter is naturally defined as multivariate generating series
    of non-decreasing functions on the graph.

    We describe the kernel of this morphism, using a simple combinatorial operation
    that we call {\em cyclic inclusion-exclusion}.
    Our result also holds for the natural noncommutative analog
    and for the commutative and noncommutative restrictions to bipartite graphs.

    An application to the theory of Kerov character polynomials is given.
\end{abstract}

\section{Introduction}
Given a poset $P=(V,<_P)$ or an acyclic directed graph $G=(V,E_G)$,
it is natural to consider the following multivariate generating function
\begin{equation}
    \label{eq:DefGamma}
    \Gamma_{P/G}(x_1,x_2,\cdots) = \sum_{\gf{f : V \to \N}{f\text{ non-decreasing}}} \prod_{v \in V} x_{f(v)} 
\end{equation}
where $\N$ is the set of positive integers and
{\em non-decreasing} means that $i<_P j$ (respectively $(i,j) \in E$) implies $f(i) \le_G f(j)$.
An example is given in \cref{sub:Gessel}.\bigskip

This is a quite classical object in the algebraic combinatorics literature:
using the terminology of the seminal book of Stanley~\cite{StanleyOrderedStructures},
the non-decreasing functions on posets correspond to $P$-partitions when $P$ has a {\em natural labelling}
(up to reversing the order of $P$).
The generating function $\Gamma_P$ has then
been considered by Gessel~\cite{GesselQSym},\
see also Stanley's textbook \cite[Section 7.19]{Stanley:EC2}.
While not symmetric in the variables $x_1,x_2,\cdots$,
this function exhibits some weaker symmetry property and 
belongs to the now well-studied algebra of {\em quasi-symmetric functions}\footnote{
In fact, the terminology {\em quasi-symmetric function} was introduced in~\cite{GesselQSym}, precisely to study $\Gamma_P$.
\label{footnote:QSym}}.

Although posets are more common objects in the literature,
the results of this paper are better formulated
in terms of acyclic directed graphs.
Obviously the map $\Gamma:G \to \Gamma_G$ defined by \eqref{eq:DefGamma} can be extended by linearity
to the vector space of formal linear combination of acyclic graphs,
that we call here the {\em graph algebra}.
A hint of the relevance of this map is the following:
there are some natural Hopf algebra structures on the graph algebras
and on quasi-symmetric functions, which turns the map $\Gamma$ into a Hopf algebra morphism,
see \cref{sub:Hopf}.
However, we shall only focus here on the linear structure.\bigskip

The main result of the present paper is a combinatorial description of the kernel
of the application $\Gamma$ from the graph algebra to quasi-symmetric functions
(Theorem~\ref{thm:Com_NonRestr}).
This description relies on a simple combinatorial operation, 
that we call {\em cyclic inclusion-exclusion}
(the definition and an example are given in \cref{sub:DefCIE}).
Before giving some background on this operation,
let us mention that this description of the kernel of $\Gamma$ is quite robust.
Indeed, we shall prove that cyclic inclusion-exclusion also describes the kernel 
of some variants of $\Gamma$, namely:
\begin{itemize}
    \item working with labeled (acyclic directed) graphs,
        it is natural to associate to them
        a multivariate generating series in {\em noncommuting variables}
        that lives in the algebra of {\em word quasi-symmetric functions} 
        \cite{NovelliThibonTrialgebras}
        (this algebra is also sometimes called {\em quasi-symmetric functions
        in noncommuting variables}, see \cite{BergeronZabrockiSymQSymFreeCofree});
        we give a description of the kernel of this application
        (denoted $\GammaNC$) in \cref{thm:NonCom_NonRestr}.
    \item We also consider restrictions of the linear maps $\Gamma$ and $\GammaNC$
        to bipartite graphs\footnote{
        A directed graph $B$ is called {\em bipartite} if its vertex set can be split as $V \sqcup W$,
        so that for each edge $(v,w) \in E$, then $v$ lies in $V$ and $w$ in $W$.
        }.
        Analogs of \cref{thm:Com_NonRestr,thm:NonCom_NonRestr} in the bipartite setting
        are given in \cref{thm:Com_Bip,thm:NonCom_Bip}.

        Note that, in the bipartite case, acyclic graphs and posets are the same objects.
        We explain below our motivation to consider such a restriction.\bigskip
\end{itemize}
In all these cases, a byproduct of our proof is
the surjectivity of the morphism $\Gamma$ (respectively $\GammaNC$
and their restriction to bipartite graphs).
The surjectivity in the commutative non restricted case was observed by Stanley
\cite[Note p7]{StanleyDescentConnectivitySets},
answering a question of Billera and Reiner.\bigskip

Our proofs use a combination
of basic linear algebra, graph combinatorics and (word) quasi-symmetric function manipulations.
In the noncommutative/labeled case,
we first exhibit a family of graphs
so that their images form a $\Z$-basis of word quasi-symmetric functions. 
Then, we show
that these graphs span the quotient of the graph algebra
by cyclic inclusion-exclusion relations.
With an easy linear algebra argument, this concludes the proof.

The commutative/unlabeled case can be obtained as a corollary
of the noncommutative/labeled case.
On the contrary, restrictions to bipartite graphs must be considered
separately from the non-restricted setting (see \cref{rem:why_non_trivial}).
The general structure of the proof is the same in the bipartite setting,
although the arguments themselves are quite different.

Along the way, this gives natural bases of the 
word quasi-symmetric function ring:
in particular, we find natural analogs
of Gessel fundamental basis \cite{GesselQSym}
and of two bases considered respectively by
R. Stanley \cite{StanleyDescentConnectivitySets}
and K. Luoto \cite{LuotoBasisQSym}.
The analog of Luoto basis has been considered recently by the author and several coauthors
in ~\cite{AvecLesJeansFPSAC}.
It could be of interest for future work on the subject,
as it has the nice property that any function $\GammaNC(B)$, where $B$ is a bipartite graph,
can be written as a multiplicity free sum of basis elements (see \cref{prop:Mult_Free_Sum}).
\bigskip

Let us now say a word about the cyclic inclusion-exclusion operation and how it has proved useful so far.

It has been introduced by the author (but not under this name)
in the article~\cite{F'eray2008} in the proof of a conjecture of Kerov
on irreducible character values of the symmetric group.
In fact, in this work, a two-alphabet variant of $\Gamma_B$
is considered for bipartite graphs $B$.
%It has then been conjectured in the author's PhD thesis \cite[Conjecture 1.2 page 27]{MyPhD}
%that cyclic inclusion-exclusion describes all relations among the functions $N_B$.
%The results of this article imply this conjecture, as any relation among the $N_B$ is also 
%a relation among the $\Gamma_B$.
%As explained in~\cite[remark on page 76]{MyPhD}, the fact that all relations between the $N_B$
%are consequences of cyclic inclusion-exclusion relations
%allow to simplify the proof of Kerov's ex-conjecture.
%%%%% THIS IS NOW EXPLAINED IN THIS PAPER %%%%%%%%%%%%%%%%%
We explain in \cref{sect:Kerov} how \cref{thm:Com_Bip} can be used to simplify
and generalize the proof of the former conjecture of Kerov.

Remarkably, this operation of cyclic inclusion-exclusion has also been fruitful in a quite
different context in~\cite{BoussicaultF'eray2009}:
the purpose of this paper was to study some rational functions considered by Greene~\cite{GreeneRationalIdentity}.
These functions are indexed by posets
and defined as sums over linear extensions of the indexing poset :
as such, they automatically verify cyclic inclusion-exclusion relations.
This gives an efficient way to compute these rational functions
and a powerful tool to study them; see~\cite{BoussicaultF'eray2009}. \bigskip

The paper is organized as follows:
\cref{sec:Prelim} introduces some standard definitions and notations.
In \cref{sec:CIE}, cyclic inclusion-exclusion is defined and it is proved
that this combinatorial construction gives some elements in the kernel of $\GammaNC$.
\cref{sec:NonRestr} deals with the non-restricted setting
and contains the proof of our main theorem in this case: the kernel of $\Gamma$
and $\GammaNC$ are spanned by the cyclic inclusion-exclusion relations
(\cref{thm:NonCom_NonRestr,thm:Com_NonRestr}).
The analogous results for the restrictions to bipartite graphs 
(\cref{thm:NonCom_Bip,thm:Com_Bip})
are established in \cref{sec:Bip}.
Finally, \cref{sect:Kerov} describes the application of \cref{thm:Com_Bip}
to the theory of Kerov character polynomials.
\section{Preliminaries}
\label{sec:Prelim}
\subsection{Labelled and unlabeled graphs}
\label{sec:graphs}
\begin{definition}
    A {\em labeled (directed) graph} $G$ is a pair $(V,E)$ where $V$ is a finite set
    and $E$ a subset of $V \times V$.

%    For an edge $e=(v_1,v_2)$ in $E$, we denote $\alpha(e)=v_1$ and $\omega(e)=v_2$ 
%    its first and second component and call them respectively {\em tail}
%    and {\em head} of the edge.
%
    A {\em directed cycle} is a list $(v_1,\dots, v_k)$ of vertices of $G$ such that
    $(v_1,v_2)$, $(v_2,v_3)$, $\cdots$, $(v_{k-1},v_k)$ and $(v_k,v_1)$
    are edges of $G$.

    A graph without directed cycles is called {\em acyclic}.
\end{definition}
For a non-negative integer $n$, we denote $[n]$ the set of positive integers smaller or equal to $n$.
In this paper, we only consider graphs with vertex set $V=[n]$, for some integer $n$.

Denote $S_n$ the group of permutations of $n$, that is of bijections from $[n]$ to $[n]$.
If $\sigma$ is a permutation of $n$ and $G$ a graph with vertex set $[n]$, 
then we consider the graph $\sigma(G)$ with vertex set $[n]$ and edge set
\[\sigma(E) = \big\{ \{\sigma(v_1),\sigma(v_2)\} \text{ with } \{v_1,v_2\} \in E \big\}.\]
\begin{definition}
    An {\em unlabeled (directed) graph} $\oG$ is an equivalence class of labeled directed graphs
    under the relation
    \[G \sim G' \Longleftrightarrow \exists \sigma \in S_n \text{ s.t. } G=\sigma(G'). \]

    As this relation preserves acyclicity of graphs,
    there is a natural notion of unlabeled acyclic graphs.
    Namely an unlabeled graph is {\em acyclic} if
    at least one (or equivalently all) labeled graph(s) in the class is(are) acyclic.
\end{definition}
We denote by $\G$ (respectively $\ov{\G}$) the vector space of linear combinations
of labeled (respectively unlabeled) acyclic graphs.
Then $\G$ (respectively $\ov{\G}$) is a graded vector space: the $d$-th homogeneous component
$\G_d$ (respectively $\ov{\G}_d$)
is by definition spanned by labeled (respectively unlabeled) graphs with vertex set $[d]$
(respectively with $d$ vertices).
The action of the symmetric group $S_d$ on graphs with vertex set $[d]$
can be extended to $\G_d$.
Then $\ov{\G}$ is the quotient of $\G$ by the vector space 
\[\{x-\sigma(x), x \in \G_d, \sigma \in S_d \text{ for some }d \ge 1\}.\]
We denote this quotient map by $\phi_u$ ($u$ stands for {\em unlabeling}).

\subsection{Quasi-symmetric functions}

As mentioned in \cref{footnote:QSym},
the ring of quasi-symmetric functions was introduced by I.~Gessel \cite{GesselQSym} 
and may be seen as a generalization of the notion of symmetric functions.
A modern introduction can be found in \cite[Section 7.19]{Stanley:EC2} or \cite[Section 3.3]{BookQSym}.

Let $n$ be a nonnegative integer.
A {\em composition} (or {\em integer composition}) of $n$ is a sequence $I=(i_1,i_2,\dots,i_r)$
of positive integers, whose sum is equal to $n$.
The notation $I\vDash n$ means that $I$ is a composition of $n$
and $\ell(I)$ denotes the number of parts of $I$.
In numerical examples, it is customary to omit parentheses and commas.
For example, $212$ is a composition of $5$.
%Given two compositions $I=(i_1,i_2,\dots,i_r)$ and $J=(j_1,j_2,\dots,j_s)$
%their {\em concatenation} is $I \cdot J=(i_1,\dots,i_r,j_1,\dots,j_s)$.

Consider the algebra
$\C[X]$ of polynomials\footnote{
Throughout the paper, we call ``polynomial in infinitely many variables'' an element of the inverse limit of
the inverse system of graded algebras
$(\C[x_1,\dots,x_n])_{n \ge 0}$ (the projection from $\C[x_1,\dots,x_n,x_{n+1}]$ to $\C[x_1,\dots,x_n]$
sends $x_{n+1}$ to $0$). In particular, it can have infinitely many monomials,
but must have a bounded degree.}
in a totally ordered alphabet of commutative variables $X=\{x_1,x_2,\dots\}.$
Monomials $X^\bv=x_1^{v_1} x_2^{v_2} \cdots$
correspond to sequences $\bv=v_1,v_2,\dots$ with finitely many
non-zero entries.
For such a sequence, we denote by $\bv_{\leftarrow}$ the list obtained by
omitting the zero entries.
\begin{definition}
A polynomial $P\in\C[X]$ is said to be {\em quasi-symmetric} if and only if
for any $\bv$ and $\bw$ such that $\bv_{\leftarrow}=\bw_{\leftarrow}$,
the coefficients of $X^\bv$ and $X^\bw$ in $P$ are equal.

One can easily prove that the set of quasi-symmetric polynomials 
is a subalgebra of $\C[X]$, called quasi-symmetric function ring and denoted
$\QSym$.
\end{definition}

It should be clear that any symmetric polynomial is quasi-symmetric.
The algebra $\QSym$ of quasi-symmetric functions has a basis of monomial
quasi-symmetric functions $(M_I)$ indexed by compositions $I= (i_1,\dots, i_r)$,
where
\begin{equation}\label{eq:QSymM}
    M_I = \sum_{k_1<\cdots<k_r} x_{k_1}^{i_1} \cdots x_{k_r}^{i_r} .
\end{equation}

In particular, the dimension of the homogeneous space $\QSym_n$ of degree $n$ 
of $\QSym$ is the number of compositions of $n$, that is $2^{n-1}$ for $n\ge 1$.

\begin{example}
    \label{Ex:MI}
    $M_{212}= \sum_{k<l<m} x_k^2 x_l x_m^2$.
\end{example}

\subsection{Word quasi-symmetric functions}
\label{sec:WQSym}
The natural noncommutative analog of $\QSym$ is the algebra of
\emph{word quasi symmetric functions}, denoted by $\WQSym$.
We recall here its construction, following the presentation of
Bergeron and Zabrocki \cite[Section 5.2]{BergeronZabrockiSymQSymFreeCofree}.
An equivalent, but slightly different presentation, using packed words instead
of set compositions, can be found in a paper of Novelli and Thibon \cite[Section 2.1]{NovelliThibonTrialgebras}.

Consider a totally ordered alphabet of noncommuting variables $\{a_1,a_2,\dots\}$.
Monomials in these variables are canonically indexed
by finite words $w$ on the alphabet $\N$ as follows
\[\bm{a}_w = a_{w_1} \, a_{w_2} \dots a_{w_{|w|}}. \]
The {\em evaluation} $\eval(w)$
of a word $w$ is the integer sequence $v=(v_1,v_2,\dots)$,
where $v_i$ is the number of letters $i$ in $w$.
Then the commutative image of $\bm{a}_w$ is $\bm{X}^{\eval(w)}$.

In the noncommutative framework, set compositions\footnote{
Set compositions are also called sometimes {\em ordered set partitions}.}
play the role of compositions.
A {\em set composition} of $n$ is an {\em (ordered)} list $\II=(I_1,\dots,I_p)$
of pairwise disjoint non-empty subsets of $\{1,\dots,n\}$,
whose union is $\{1,\dots,n\}$.
In numerical example, we sort integers inside a part and use a vertical
line to separate the parts.
For example, the set composition $(\{1,5\},\{3,4,6\},\{2\})$
is denoted $15|346|2$.

To a word $w$ on the (ordered) alphabet $\N$ of length $\ell$,
we associate the set composition $\II=\Delta(w)$ such that 
$j \in I_{|\{w_r:w_r \le w_i\}|}$ (for every $j$ in $[\ell]$).
For example $\Delta(275525)=15|346|2$.
\begin{definition}
    A polynomial\footnote{
As in the commutative setting, polynomials in infinitely may variables should be formally defined
as inverse limit of a sequence of polynomials in finitely many variables.
} in noncommuting variables $a_1,a_2,\dots$
is a {\em word quasi symmetric function} if and only if
$a_v$ and $a_w$ are equal
as soon as $\Delta(v)$ and $\Delta(w)$ coincide.
\end{definition}

One can easily prove that the set $\WQSym$ of word quasi symmetric functions is 
an algebra.
A linear basis of $\WQSym$ is given as follows:
\[ \M_\II = \sum_{\gf{w \text{ s.t.}}{\Delta(w)=\II}} \bm{a}_w. \]

Clearly, if we only remember the sizes of the sets in a  set composition $\II$,
we get an integer composition that we denote $\comm(\II)$ ($c$ stands for commuting).
For example, $\comm(15|346|2)=231$.
With this notation, the commutative image of $\M_\II$ is $M_{\comm(\II)}$.
Therefore, sending the variables $a_1,a_2,\dots$ to their commutative analogs $x_1,x_2,\dots$
defines a surjective projection from $\WQSym$ to $\QSym$, that we abusively also denote $\comm$.
This projection can be alternatively realized as follows: the symmetric group $S_n$ acts on
the homogeneous component $\WQSym_n$ of $\WQSym$ of degree $n$ by permuting factors in every monomial.
Then $\QSym$ is the quotient of $\WQSym$ by the ideal spanned linearly by
\[\{x-\sigma(x), x \in \WQSym_d, \sigma \in S_d \text{ for some }d \ge 1\}.\]

To finish, let us mention that 
the ordered Bell numbers~\cite[A000670]{OEIS} count set compositions of $[n]$,
and thus give the dimension of the homogeneous subspace of degree $n$ of $\WQSym$.

\begin{example}
    Consider the set composition $\bm{I}=25|4|13$.
    Its evaluation is the integer composition $212$.
    Then the associate basis element of $\WQSym$ is
    \[\M_\II= \sum_{k<l<m} a_m \, a_k \, a_m \, a_l \, a_k.\]
    It is easy to check that its commutative image is $M_{212}$
    (given in \cref{Ex:MI}), as claimed.
\end{example}

\subsection{Gessel's morphism}
\label{sub:Gessel}
\begin{definition}
    Let $G$ be a graph on vertex set $[n]$.
    A function $f: [n] \to \N$ is called {\em $G$ non-decreasing}
    if, for any edge $(i,j)$ in $E$, one has $f(i) \le f(j)$.

    For a labeled graph $G$, we define $\GammaNC(G)$ as
    \[\GammaNC(G) := \sum_{\gf{f: [n] \to \N}{f\ G \text{ non-decreasing}}}
    a_{f(1)} \dots a_{f(n)}.\]
\end{definition}
\begin{example}
    Consider the graph $G=\begin{array}{c}\begin{tikzpicture}[font=\tiny,scale=.3]
        \tikzstyle{vertex}=[circle,draw=black,inner sep=.3pt];
        \node[vertex] (v3) at (0,0) {$3$};
        \node[vertex] (v1) at (2,0) {$1$};
        \node[vertex] (v2) at (1,1.5) {$2$};
        \node[vertex] (v4) at (3,1.5) {$4$};
        \draw[->] (v3) -- (v2);
        \draw[->] (v1) -- (v2);
        \draw[->] (v1) -- (v4);
    \end{tikzpicture}\end{array}$, then
    \[\GammaNC(G)= \! \! 
    \sum_{\gf{k_1,k_2,k_3,k_4}{k_3 \le k_2,\ k_1 \le k_2,\ k_1 \le k_4}}  \! \! a_{k_1} a_{k_2} a_{k_3} a_{k_4}. \]
\end{example}

It is clear that $\GammaNC(G)$ is a word quasi-symmetric function.
Therefore, $\GammaNC$ extends as a linear application from
$\G$ to $\WQSym$.\medskip

The image $\comm(\GammaNC(G))$ of $\GammaNC(G)$ in $\QSym$
does not change if we replace $G$ by an isomorphic labeled graph $G'=\sigma(G)$.
Thus the morphism \[\comm \circ \GammaNC : \G \to \QSym\] factorizes through the quotient
$\ov{\G}$ and defines a morphism $\ov{\G} \to \QSym$.
We recover of course the morphism $\Gamma$ defined by \cref{eq:DefGamma}
in the introduction and studied by Gessel in \cite{GesselQSym}.

In other words, we have the commutative diagram
\[ \begin{tikzpicture}[auto]
    \node (G) at (0,0) {$\G$};
    \node (oG) at (0,-2) {$\ov{\G}$};
    \node (WQ) at (3,0) {$\WQSym$};
    \node (Q) at (3,-2) {$\QSym$};
    \draw[->] (G) to node {$\GammaNC$} (WQ);
    \draw[->] (G) to node [swap] {$\phi_u$} (oG);
    \draw[->] (WQ) to node {$\phi_c$} (Q);
    \draw[->] (oG) to node [swap] {$\Gamma$} (Q);
\end{tikzpicture}.\]

\subsection{Hopf algebra structures}
\label{sub:Hopf}
In this Section, we mention known Hopf algebra structures
of the spaces $\G$, $\ov{G}$, $\QSym$ and $\WQSym$ which turns
the morphisms described above into Hopf algebra morphisms.

As we focus in this paper on linear structures, this material won't be used 
and is only presented as additional motivation.
This explains the lack of details and examples in this Section.\medskip

The space $\G$ has a Hopf algebra structure with the following product and coproduct:
\begin{itemize}
    \item The product of $G$ and $G'$ is $G \sqcup (G')^{\uparrow |G|}$, where
        $(G')^{\uparrow |G|}$ means that we have shifted all vertex labels in $G'$ by the number $|G|$ of vertices
        of $G$, so that the disjoint union is a graph with vertex set $[|G|+|G'|]$.
    \item The coproduct of a graph $G$ with vertex set $[n]$ is given by
        \[ \Delta(G) = \sum_I \std(G[I]) \times \std\big(G\big[[n] \backslash I \big]\big),\]
        where the sum runs over subsets $I$ of $[n]$ such that there is no edges going from $[n] \backslash I$ to $I$.
        Here, $G[I]$ and $G\big[ [n] \backslash I \big]$ denote 
        the graphs induced by $G$ on $I$ and $[n] \backslash I$
        and $\std(H)$ consists in relabelling vertices of $H$ in an order-preserving way so that the
        result has vertex set $[m]$ for some integer $m$.
\end{itemize}
These operations are compatible with the action of symmetric groups
(in some sense that has to be precised)
and thus are also naturally defined on the quotient $\ov{\G}$.

The spaces $\QSym$ and $\WQSym$ have natural algebra structures inherited from the polynomial algebras,
in which they live.
It is also possible to define these products on the basis by some combinatorial operations
on integer compositions and set compositions.

The coproducts of $\QSym$ and $\WQSym$ are given on the bases by the formulas:
\begin{align*}
    \Delta(M_I)&=\sum_{k=0}^{\ell(I)} M_{(i_1,\dots,i_k)}\otimes M_{(i_{k+1},\dots,i_{\ell(I)})};\\
    \Delta(\M_\II)&=\sum_{k=0}^{\ell(\II)} \M_{(I_1,\dots,I_k)} \otimes \M_{(I_{k+1},\dots,I_{\ell(\II)})}.
\end{align*}

It is not difficult to check that these multiplication and comultiplication structures
are compatible with all morphisms from the previous Section.

\begin{remark}
    A detailed description of the Hopf algebra structure of $\QSym$ can be found for example in \cite[Section 3.3]{BookQSym}.
    For $\WQSym$, we refer to \cite[Section 2.1]{NovelliThibonTrialgebras}.

    The Hopf algebra structure presented here for acyclic graphs is similar to the one
    considered on posets by Aguiar and Mahajan in \cite[Section 13.1]{AguiarBook} with the formalism of Hopf monoids.
    It should be stressed that this Hopf algebra structure is different from the so-called
    {\em incidence Hopf algebra}, another Hopf algebra on posets considered in the litterature,
    see {\em e.g.} \cite{EhrenborgQSymPosets} and references therein.
\end{remark}

\section{Cyclic inclusion-exclusion}
\label{sec:CIE}
\subsection{Definition and example}
\label{sub:DefCIE}
Let $G$ be a directed graph.
Consider $G$ as a non directed graph and assume that it contains a cycle $C$.

Formally, such a cycle $C$ is
a list
$C=(x_1,\dots,x_k)$
such that, for $1\leq i \leq k$,
\begin{itemize}
    \item either $(x_i,x_{i+1})$ is an edge of $G$;
    \item or $(x_{i+1},x_i)$ is an edge of $G$,
\end{itemize}
where, by convention, $x_{k+1}:=x_1$.
In the first case, we say that $(x_i,x_{i+1})$ is in a set $C^+$.
In the second case, we say that $(x_i,x_{i+1})$ is in $C^-$.

Another description of the sets $C^+$ and $C^-$ is the following.
Edges of $C$ have two orientations:
\begin{itemize}
    \item their orientation in the cycle $C$;
    \item and their orientation as edges of $G$.
\end{itemize}
We denote $C^+$ (respectively $C^-$) the set of edges of $C$,
for which these two orientations coincide (respectively do not coincide).

Finally, we define the following element of the graph algebra $\G$:
\[\CIE_{G,C} = \sum_{D \subseteq C^+} (-1)^{|D|} G \setminus D,\]
where $G \setminus D$ is the (directed acyclic) graph obtained from $G$
by erasing the edges in $D$
(and keeping the same set of vertices).

\begin{example}
    \label{ex:Graph_for_CIE}
    Consider the graph $G_\ex$ from \cref{fig:Graph_for_CIE}.
    The non-oriented version of $G_\ex$ contains several cycles, among them
    $C_\ex=(6,2,3,5,1)$.
    This cycle is represented as a subgraph of $G_\ex$ in \cref{fig:Graph_for_CIE}
    with the two orientations described above.
    Then the set $C_\ex^+$ is equal to $\{ (6,2),(2,3),(3,5) \}$
    and $\CIE_{G_\ex,C_\ex}$ is given in \cref{fig:Graph_for_CIE}.

    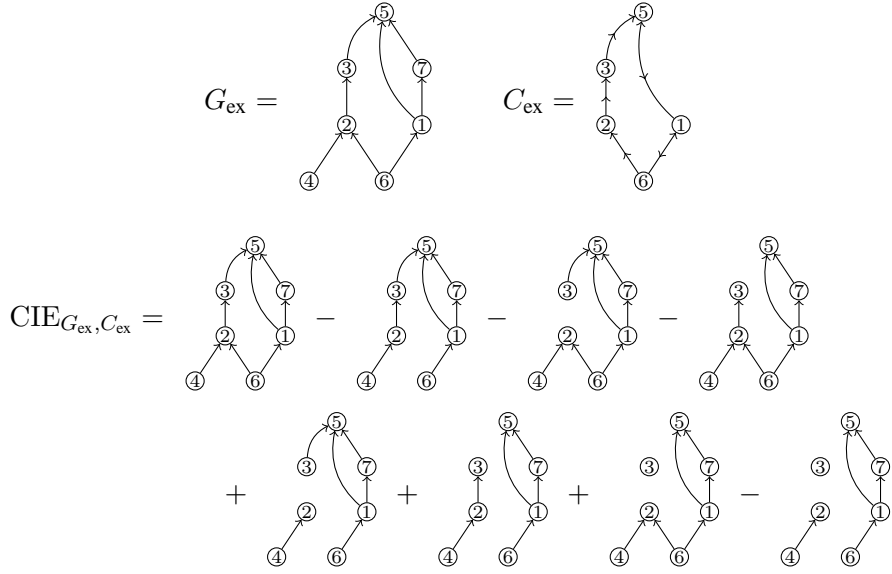
\begin{figure}[t]
                \[
        G_\ex=\begin{array}{c}\begin{tikzpicture}[font=\tiny,scale=.5]
        \tikzstyle{vertex}=[circle,draw=black,inner sep=.3pt];
        \node[vertex] (v3) at (0,0) {$4$};
        \node[vertex] (v1) at (2,0) {$6$};
        \node[vertex] (v2) at (1,1.5) {$2$};
        \node[vertex] (v4) at (3,1.5) {$1$};
        \node[vertex] (v5) at (1,3) {$3$};
        \node[vertex] (v6) at (3,3) {$7$};
        \node[vertex] (v7) at (2,4.5) {$5$};
        \draw[->] (v3) -- (v2);
        \draw[->] (v1) -- (v2);
        \draw[->] (v1) -- (v4);
        \draw[->] (v2) -- (v5);
        \draw[->] (v5) to [bend left] (v7);
        \draw[->] (v4) to [bend left] (v7);
        \draw[->] (v4) -- (v6);
        \draw[->] (v6) -- (v7);
       \end{tikzpicture}\end{array}
       \qquad
        C_\ex=\begin{array}{c}\begin{tikzpicture}[font=\tiny,scale=.5]
        \tikzstyle{vertex}=[circle,draw=black,inner sep=.3pt];
        \node[vertex] (v1) at (2,0) {$6$};
        \node[vertex] (v2) at (1,1.5) {$2$};
        \node[vertex] (v4) at (3,1.5) {$1$};
        \node[vertex] (v5) at (1,3) {$3$};
        \node[vertex] (v7) at (2,4.5) {$5$};
        \begin{scope}[decoration={markings,           
                        mark=at position .5 with {\arrow{>}}}]
        \draw[->,postaction={decorate}] (v1) -- (v2);
        \draw[->,postaction={decorate}] (v2) -- (v5);
        \draw[->,postaction={decorate}] (v5) to [bend left] (v7);
    \end{scope}
        \begin{scope}[decoration={markings,           
                        mark=at position .5 with {\arrow{<}}}]
        \draw[->,postaction={decorate}] (v1) -- (v4);
        \draw[->,postaction={decorate}] (v4) to [bend left] (v7);
    \end{scope}
       \end{tikzpicture}\end{array}
        \]
        \begin{multline*}
            \CIE_{G_\ex,C_\ex}=\begin{array}{c}\begin{tikzpicture}[font=\tiny,scale=.4]
        \tikzstyle{vertex}=[circle,draw=black,inner sep=.3pt];
        \node[vertex] (v3) at (0,0) {$4$};
        \node[vertex] (v1) at (2,0) {$6$};
        \node[vertex] (v2) at (1,1.5) {$2$};
        \node[vertex] (v4) at (3,1.5) {$1$};
        \node[vertex] (v5) at (1,3) {$3$};
        \node[vertex] (v6) at (3,3) {$7$};
        \node[vertex] (v7) at (2,4.5) {$5$};
        \draw[->] (v3) -- (v2);
        \draw[->] (v1) -- (v2);
        \draw[->] (v1) -- (v4);
        \draw[->] (v2) -- (v5);
        \draw[->] (v5) to [bend left] (v7);
        \draw[->] (v4) to [bend left] (v7);
        \draw[->] (v4) -- (v6);
        \draw[->] (v6) -- (v7);
       \end{tikzpicture}\end{array}
       -\begin{array}{c}\begin{tikzpicture}[font=\tiny,scale=.4]
        \tikzstyle{vertex}=[circle,draw=black,inner sep=.3pt];
        \node[vertex] (v3) at (0,0) {$4$};
        \node[vertex] (v1) at (2,0) {$6$};
        \node[vertex] (v2) at (1,1.5) {$2$};
        \node[vertex] (v4) at (3,1.5) {$1$};
        \node[vertex] (v5) at (1,3) {$3$};
        \node[vertex] (v6) at (3,3) {$7$};
        \node[vertex] (v7) at (2,4.5) {$5$};
        \draw[->] (v3) -- (v2);
 %       \draw[->] (v1) -- (v2);
        \draw[->] (v1) -- (v4);
        \draw[->] (v2) -- (v5);
        \draw[->] (v5) to [bend left] (v7);
        \draw[->] (v4) to [bend left] (v7);
        \draw[->] (v4) -- (v6);
        \draw[->] (v6) -- (v7);
       \end{tikzpicture}\end{array}
       -\begin{array}{c}\begin{tikzpicture}[font=\tiny,scale=.4]
        \tikzstyle{vertex}=[circle,draw=black,inner sep=.3pt];
        \node[vertex] (v3) at (0,0) {$4$};
        \node[vertex] (v1) at (2,0) {$6$};
        \node[vertex] (v2) at (1,1.5) {$2$};
        \node[vertex] (v4) at (3,1.5) {$1$};
        \node[vertex] (v5) at (1,3) {$3$};
        \node[vertex] (v6) at (3,3) {$7$};
        \node[vertex] (v7) at (2,4.5) {$5$};
        \draw[->] (v3) -- (v2);
        \draw[->] (v1) -- (v2);
        \draw[->] (v1) -- (v4);
 %       \draw[->] (v2) -- (v5);
        \draw[->] (v5) to [bend left] (v7);
        \draw[->] (v4) to [bend left] (v7);
        \draw[->] (v4) -- (v6);
        \draw[->] (v6) -- (v7);
       \end{tikzpicture}\end{array}
       -\begin{array}{c}\begin{tikzpicture}[font=\tiny,scale=.4]
        \tikzstyle{vertex}=[circle,draw=black,inner sep=.3pt];
        \node[vertex] (v3) at (0,0) {$4$};
        \node[vertex] (v1) at (2,0) {$6$};
        \node[vertex] (v2) at (1,1.5) {$2$};
        \node[vertex] (v4) at (3,1.5) {$1$};
        \node[vertex] (v5) at (1,3) {$3$};
        \node[vertex] (v6) at (3,3) {$7$};
        \node[vertex] (v7) at (2,4.5) {$5$};
        \draw[->] (v3) -- (v2);
        \draw[->] (v1) -- (v2);
        \draw[->] (v1) -- (v4);
        \draw[->] (v2) -- (v5);
 %       \draw[->] (v5) to [bend left] (v7);
        \draw[->] (v4) to [bend left] (v7);
        \draw[->] (v4) -- (v6);
        \draw[->] (v6) -- (v7);
       \end{tikzpicture}\end{array}\\
       +\begin{array}{c}\begin{tikzpicture}[font=\tiny,scale=.4]
        \tikzstyle{vertex}=[circle,draw=black,inner sep=.3pt];
        \node[vertex] (v3) at (0,0) {$4$};
        \node[vertex] (v1) at (2,0) {$6$};
        \node[vertex] (v2) at (1,1.5) {$2$};
        \node[vertex] (v4) at (3,1.5) {$1$};
        \node[vertex] (v5) at (1,3) {$3$};
        \node[vertex] (v6) at (3,3) {$7$};
        \node[vertex] (v7) at (2,4.5) {$5$};
        \draw[->] (v3) -- (v2);
 %       \draw[->] (v1) -- (v2);
        \draw[->] (v1) -- (v4);
 %       \draw[->] (v2) -- (v5);
        \draw[->] (v5) to [bend left] (v7);
        \draw[->] (v4) to [bend left] (v7);
        \draw[->] (v4) -- (v6);
        \draw[->] (v6) -- (v7);
       \end{tikzpicture}\end{array}
       +\begin{array}{c}\begin{tikzpicture}[font=\tiny,scale=.4]
        \tikzstyle{vertex}=[circle,draw=black,inner sep=.3pt];
        \node[vertex] (v3) at (0,0) {$4$};
        \node[vertex] (v1) at (2,0) {$6$};
        \node[vertex] (v2) at (1,1.5) {$2$};
        \node[vertex] (v4) at (3,1.5) {$1$};
        \node[vertex] (v5) at (1,3) {$3$};
        \node[vertex] (v6) at (3,3) {$7$};
        \node[vertex] (v7) at (2,4.5) {$5$};
        \draw[->] (v3) -- (v2);
  %      \draw[->] (v1) -- (v2);
        \draw[->] (v1) -- (v4);
        \draw[->] (v2) -- (v5);
  %      \draw[->] (v5) to [bend left] (v7);
        \draw[->] (v4) to [bend left] (v7);
        \draw[->] (v4) -- (v6);
        \draw[->] (v6) -- (v7);
       \end{tikzpicture}\end{array}
       +\begin{array}{c}\begin{tikzpicture}[font=\tiny,scale=.4]
        \tikzstyle{vertex}=[circle,draw=black,inner sep=.3pt];
        \node[vertex] (v3) at (0,0) {$4$};
        \node[vertex] (v1) at (2,0) {$6$};
        \node[vertex] (v2) at (1,1.5) {$2$};
        \node[vertex] (v4) at (3,1.5) {$1$};
        \node[vertex] (v5) at (1,3) {$3$};
        \node[vertex] (v6) at (3,3) {$7$};
        \node[vertex] (v7) at (2,4.5) {$5$};
        \draw[->] (v3) -- (v2);
        \draw[->] (v1) -- (v2);
        \draw[->] (v1) -- (v4);
  %      \draw[->] (v2) -- (v5);
  %      \draw[->] (v5) to [bend left] (v7);
        \draw[->] (v4) to [bend left] (v7);
        \draw[->] (v4) -- (v6);
        \draw[->] (v6) -- (v7);
       \end{tikzpicture}\end{array}
       -\begin{array}{c}\begin{tikzpicture}[font=\tiny,scale=.4]
        \tikzstyle{vertex}=[circle,draw=black,inner sep=.3pt];
        \node[vertex] (v3) at (0,0) {$4$};
        \node[vertex] (v1) at (2,0) {$6$};
        \node[vertex] (v2) at (1,1.5) {$2$};
        \node[vertex] (v4) at (3,1.5) {$1$};
        \node[vertex] (v5) at (1,3) {$3$};
        \node[vertex] (v6) at (3,3) {$7$};
        \node[vertex] (v7) at (2,4.5) {$5$};
        \draw[->] (v3) -- (v2);
 %       \draw[->] (v1) -- (v2);
        \draw[->] (v1) -- (v4);
 %       \draw[->] (v2) -- (v5);
 %       \draw[->] (v5) to [bend left] (v7);
        \draw[->] (v4) to [bend left] (v7);
        \draw[->] (v4) -- (v6);
        \draw[->] (v6) -- (v7);
    \end{tikzpicture}\end{array}\end{multline*}
        \caption{Graph $G_\ex$, cycle $C_\ex$ and the graph algebra element $\CIE_{G_\ex,C_\ex}$
        from \cref{ex:Graph_for_CIE}.}
        \label{fig:Graph_for_CIE}
    \end{figure}
\end{example}

\subsection{Cyclic inclusion-exclusion relations}
\begin{proposition}\label{prop:CIE_in_Kernel}
    For any graph $G$ and cycle $C$ of $G$, one has:
    \[\GammaNC(\CIE_{G,C})=0.\]
\end{proposition}

\begin{proof}
Let $n$ be the size of $G$.
Using the definitions of the morphism $\GammaNC$ and of the element $\CIE_{G,C}$, one has:
\begin{multline*}
    \GammaNC(\CIE_{G,C}) = \sum_{ D \subseteq C^+} (-1)^{|D|} \left[ 
    \sum_{\gf{f : [n] \to \N}{f \ (G \setminus D) \text{ non-decreasing} }}
   a_{f(1)} \cdots a_{f(n)}  \right] \\
    = \sum_{f:[n] \to \N} \left( a_{f(1)} \cdots a_{f(n)} \right)
    \left( \sum_{ D \subseteq C^+} (-1)^{|D|} 
    \big[ f\ (G \setminus D) \text{ non-decreasing}\big] \right),
\end{multline*}
where $[\text{condition}]$ is $1$ if the condition is fulfilled and $0$ else.
The idea of the proof is to show that for any function $f:[n] \to \N$,
its contribution
\begin{equation}
    \sum_{ D \subseteq C^+} (-1)^{|D|}                        
        \big[ f\ (G \setminus D) \text{ non-decreasing}\big]
    \label{EqContributionF}
\end{equation}
is zero.\medskip

If $f$ is not a $G \backslash C^+$ non-decreasing function,
then each summand of \eqref{EqContributionF} is zero
and the conclusion holds trivially in this case.\medskip

Let $f:[n] \to \N$ be a $G \backslash C^+$ non-decreasing function,
define
\[D_f=\big\{(x,y) \in C^+ \text{ s.t. } f(x) > f(y) \big\} \ \subseteq C^+.\]
It is straightforward that $D_f$ fulfills the following property:
\begin{equation}
\forall D \subseteq C^+, \ \ \ \ f \text{ is $G \setminus D$ non-decreasing}
\Longleftrightarrow D_f \subseteq D.
    \label{EqDfIsCool}
\end{equation}
Hence \cref{EqContributionF} can be rewritten as:
    \[ \sum_{D_f \subseteq D \subseteq C^+} (-1)^{|D|}. \]
which is equal to zero if and only if $D_f \neq C^+$.
Therefore, to end the proof of the proposition, it is enough
to show that, for any $G \backslash C^+$ non-decreasing function,
$D_f$ is strictly included in $C^+$.\medskip

We proceed by contradiction.
Suppose that we can find a $G \backslash C^+$ non-decreasing function $f$ for which $D_f=C^+$.
This means that, for each $(x,y)$ in $C^+$, one has $f(x) > f(y)$.

Besides, since $f$ is a $G \backslash C^+$ non-decreasing function,
one has $f(x) \leq f(y)$ for any edge $(x,y)$ of $G$ which is not in $C^+$, 
so in particular for any couple $(y,x)$ in $C^-$.

Recall now that $C$ is a cycle in the undirected version of $G$.
Formally, $C$ is a list $(x_1,\dots,x_k)$
such that, for $1\leq i \leq \ell$, (by convention, $x_{k+1}=x_1$)
\begin{itemize}
    \item either $(x_i,x_{i+1})$ is an edge of $G$ and $(x_i,x_{i+1}) \in C^+$;
    \item or $(x_{i+1},x_i)$ is an edge of $G$ and $(x_i,x_{i+1}) \in C^-$.
\end{itemize}
Using the remarks above, we can conclude in both cases that 
$f(x_i) \geq f(x_{i+1})$. Bringing everything together,
\[f(x_1) \geq f(x_2) \geq \dots \geq f(x_{\ell}-1) \geq f(x_\ell) \geq f(x_1).\]
As $C^+$ can not be empty (otherwise, $(x_k,\dots,x_1)$ would be a directed cycle),
at least one of these inequalities should be strict.
We have reached a contradiction
and $D_f$ must be strictly included in $C^+$.
\end{proof}

\cref{prop:CIE_in_Kernel} gives some relations between the 
word quasi-symmetric functions $\GammaNC(G)$.
We call these relations {\em cyclic inclusion-exclusion relations} 
(CIE relations for short).
Formally, the elements $(\CIE_{G,C})$ span linearly a subspace,
that we shall denote $\ideal$, which is included in the kernel of $\GammaNC$.

We shall prove in the next Section that any relation among
the $\GammaNC(G)$ can be deduced from CIE relations.
In other terms, the space $\ideal$ is exactly the kernel of $\GammaNC$.
We will also prove that analog results hold for some quotients/restrictions of
$\GammaNC$.

\begin{special_case}
    We describe here the special case where $|C^+|=1$.
    If $e=(v_1,v_2)$ is the element of $C^+$, this means that the graph $G$ contains another path\footnote{
    A path form $x$ to $y$ is a list $(v_0,v_1,\ldots,v_k)$ with $v_0=x$ and $v_k=y$ such that
    for every $i$ in $[k]$, the pair $(v_{i-1},v_i)$ is an edge of $G$.
    \label{Footnote:path}}
    from $v_1$ to $v_2$.
    Informally, $e$ can be obtained from other edges of $G$ by {\em transitivity}.

    In this case, the inclusion-exclusion relation yields $\GammaNC(G)=\GammaNC(G \setminus \{e\})$.
    This is indeed true, as non-decreasing functions on both graphs are the same.
    \label{Special_Case:Transitivity}
\end{special_case}

\begin{remark}
A weaker form of \cref{prop:CIE_in_Kernel} (in the commutative setting)
has been established in \cite[Theorem 4.1]{BoussicaultF'eray2009}
and widely used to extend some rational identity due to Greene \cite{GreeneRationalIdentity}.
The structure of the proof is exactly the same.
\end{remark}
\label{sec:CIE_Rel}

\section{The kernel in the non-restricted case}
\label{sec:NonRestr}
\subsection{The graphs $G_\II$}
\label{sec:GI}

\begin{definition}
Let $\II=(I_1,\dots,I_r)$ be a set composition of $[n]$.
We consider the directed graph $G_\II$ with vertex set $[n]$
and edge set
\[ \bigsqcup_{j < k} I_j \times I_k. \]
In other terms, there is an edge between $x$ and $y$
if the index of the set of $\II$ containing $x$ is smaller than
the one of the set containing $y$.
\end{definition}

\begin{example}
    \label{ex:GI}
    Take $\II_\ex=15|346|2$. Then $G_{\II_\ex}$
     and the associated word quasi symmetric function are
     \begin{equation}
         G_{\II_\ex}=\begin{array}{c}
         \begin{tikzpicture}[font=\tiny,scale=.35]
        \tikzstyle{vertex}=[circle,draw=black,inner sep=.3pt];
        \node[vertex] (v1) at (0,0) {$1$};
        \node[vertex] (v5) at (2,0) {$5$};
        \node[vertex] (v3) at (-1,1.5) {$3$};
        \node[vertex] (v4) at (1,1.5) {$4$};
        \node[vertex] (v6) at (3,1.5) {$6$};
        \node[vertex] (v2) at (1,3) {$2$};
        \draw[->] (v1) -- (v3);
        \draw[->] (v1) -- (v4);
        \draw[->] (v1) -- (v6);
        \draw[->] (v5) -- (v3);
        \draw[->] (v5) -- (v4);
        \draw[->] (v5) -- (v6);
        \draw[->] (v3) -- (v2);
        \draw[->] (v4) -- (v2);
        \draw[->] (v6) -- (v2);
        \draw[->] (v1) to [bend left=20] (v2);
        \draw[->] (v5) to [bend right=20] (v2);
        \end{tikzpicture}
     \end{array}; \quad
     \GammaNC(G_{\II_\ex}) = \!\! \sum_{ \gf{k_1,\dots,k_6}{
    \gf {\max(k_1,k_5) \le \min(k_3,k_4,k_6) }{\max(k_3,k_4,k_6) \le k_2}}}
    \!\! a_{k_1}  \cdots  a_{k_6}.
    \label{eq:GamGIex}
\end{equation}
\end{example}

\subsection{A $\Z$-basis of $\WQSym$}
\label{sec:Zbasis_NonRes}
The purpose of this Section is to prove that $\GammaNC(G_\II)$
is a $\Z$-basis of $\WQSym$.
The proof requires to consider two additional bases of $\WQSym$
and prove that three change of basis matrices are unitriangular 
(with respect to different orders of the basis elements).
\medskip

As in \cite[Section 3.1]{WilsonExtensxionMacMahon},
it will be convenient to work with descent-starred permutations instead of set compositions.
\begin{definition}
    We call {\em descent-starred permutation} a couple $(\sigma,D)$ such
    that $D$ is a subset of the descent set
    $\{i, \sigma(i) > \sigma(i+1) \}$ of $\sigma$.

    The descents in $D$ are termed {\em starred}.
\end{definition}
In numerical example, we represent a descent-starred permutation $(\sigma,D)$
by the word notation of $\sigma$ in which the elements of index in $D$
are followed by a star.
For example the descent-starred permutation $(3142,\{3\})$ will be denoted
$31\u{4}2$.

\begin{lemma}
    Descent-starred permutations of $n$ are in bijection with set compositions of $[n]$.
\end{lemma}
\begin{proof}
    From the numerical notation of a set composition $\II$,
    we sort each part in decreasing order and remove vertical bars to get the
    word notation of $\sigma$.
    Then mark with a star the descents inside the same part of $\II$.
    This is clearly a bijection.
\end{proof}
For example, the descent-starred permutation associated to $15|346|2$ is $\u{5}1\u{6}\u{4}32$.

Let us define three families of word quasi-symmetric functions indexed by descent-starred permutations
$\M_{(\sigma,D)}$, $\L_{(\sigma,D)}$ and $\F_{(\sigma,D)}$.
All of them are defined as a sum
\[\sum a_{k_1} \cdots a_{k_n}\]
over lists $\kk=(k_1,\dots,k_n)$ of positive integers with conditions 
given in the following table (for integers $x$ in $[n-1]$):
\[\begin{array}{|c|c|c|c|}
    \hline  & \M_{(\sigma,D)} & \L_{(\sigma,D)} & \F_{(\sigma,D)} \\ \hline
    x \in D & k_{\sigma(x)} = k_{\sigma(x+1)} & k_{\sigma(x)} = k_{\sigma(x+1)} & k_{\sigma(x)} < k_{\sigma(x+1)} \\ \hline
    x \notin D & k_{\sigma(x)} < k_{\sigma(x+1)} & k_{\sigma(x)} \le k_{\sigma(x+1)} & k_{\sigma(x)} \le k_{\sigma(x+1)} \\ \hline
\end{array} \]\medskip

In the definitions of $\M_{(\sigma,D)}$ and $\L_{(\sigma,D)}$,
we require that $k_{\sigma(x)} = k_{\sigma(x+1)}$ for $x \in D$,
which implies that the function $x \mapsto k_x$ should be constant
on the parts of the associated set composition $\II$.
Moreover, in $\M_{(\sigma,D)}$, together the strict inequalities for $x \notin D$,
this is equivalent to $\Delta(\kk)=\II$, so that we have
$\M_{(\sigma,D)}=M_\II$.

\begin{remark}
    The commutative projection of $\F_{(\sigma,D)}$ is $F_{J}$,
    where $F$ is the so-called {\em fundamental basis} of $\QSym$
    and $J$ the (integer) composition associated with the set $D$
    (we use here the terminology of \cite[Section 7.19]{Stanley:EC2}).
\end{remark}

\begin{lemma}
    \label{lem:Bij_MP}
    The families $(\L_{(\sigma,D)})$ and $(\F_{(\sigma,D)})$, indexed by descent-starred permutations,
    are $\Z$-basis of $\WQSym$.
\end{lemma}
\begin{proof}
    We start by recalling some classical terminology:
    we say that a set-partition $\II$ is finer than $\JJ$ and denote $\II \triangleleft \JJ$
    if $\JJ$ can be obtained from $\II$ by removing vertical lines and reordering the blocks: for example,
    $15|346|2$ is finer than $13456|2$ and than $15|2346$.
    
    Let $(\sigma,D)$ be a descent-starred permutation and $\II$ the associated set composition.
    Using the remark above, the definition of $\L_{(\sigma,D)}$ 
    (that we will also denote $\L_\II$) can be rewritten as
    \[\L_\II = \sum a_{k_1} \cdots a_{k_n},\]
    where the sum runs over lists $(k_1,\dots,k_n)$ that are constant
    on the parts of $I$ and such that the value of $k$ on $I_m$ is 
    at most the one on $I_{m+1}$ (for each $m$ in $[\ell(I)-1]$).
If we cut the sum depending on which indices $i_\ell$ are equal, we obtain\footnote{
See \cite[Eq. (2)]{GesselQSym} for the commutative analog of this statement.}
\[\L_\II = \sum_{J \triangleright I} \M_\JJ.\]
 This implies that $(\L_\II)$ is a $\Z$-basis of $\WQSym$
 as its matrix in the basis $(\M_\II)$ is unitriangular. \medskip

 Consider now the family $\F_{(\sigma,D)}$.
 We first rewrite the definitions of $\F_{(\sigma,D)}$ and $\L_{(\sigma,D)}$ as follows:
 \begin{align}
     \F_{(\sigma,D)} &= \sum a_{k_1} \cdots a_{k_n} 
     \prod_{x \in D} \big(1-\delta_{k_{\sigma(x)},k_{\sigma(x+1)}}\big),
     \label{eq:F_rewritten}\\
     \L_{(\sigma,D)} &= \sum a_{k_1} \cdots a_{k_n} 
     \prod_{x \in D} \big(\delta_{k_{\sigma(x)},k_{\sigma(x+1)}}\big),
 \end{align}
 where both sums run over lists $(k_1,\cdots,k_n)$ that satisfy 
 $k_{\sigma(1)} \le \cdots \le k_{\sigma(n)}$ and $\delta_{i,j}$ is the usual Kronecker symbol.
 Expanding the product in \eqref{eq:F_rewritten}, we get
\[ \F_{(\sigma,D)} = \sum_{D' \subseteq D} (-1)^{|D'|} \L_{(\sigma,D')}. \]
Hence the matrix of the family $\F_{(\sigma,D)}$ in the basis $\L_{(\sigma,D)}$
is unitriangular with respect to the following order\footnote{
This order is isomorphic to the order on set compositions
denoted $\leq_\star$ in \cite[Section 6]{BergeronZabrockiSymQSymFreeCofree}.}:
\[(\sigma',D') \leq_1 (\sigma,D) \Rightarrow \begin{cases}
    \sigma = \sigma' \\
    D' \subseteq D
\end{cases} \]
This proves that $(\F_{(\sigma,D)})$ is a $\Z$-basis of $\WQSym$.
\end{proof}\medskip

We now explain how $\GammaNC(G_\II)$ writes on the $\F$ basis.
If $\II=(I_1,\dots,I_r)$ is a set composition, we consider the following set $\MP(\II)$ of descent-starred permutations:
\begin{itemize}
    \item As a word $\sigma=w_1 \cdots w_r$, where $w_m$ contains exactly once each element of $I_m$ ;
    \item The descent in position $x$ is starred if $\sigma_x$ and $\sigma_{x+1}$ are in the same
        part of $\II$. In other words, for each $m$, we mark the descents in $w_m$,
        but not the potential descent created by concatenating $w_m$ and $w_{m+1}$.
\end{itemize}
For example, take $\II_\ex=15|346|2$, then $\MP(\II_\ex)$ contains the following 12 descent-starred permutations:
\begin{multline*}
    153462,\ \u{5}13462,\ 15\u{4}362,\ \u{5}1\u{4}362,\ 15\u{6}\u{4}32,\ \u{5}1\u{6}\u{4}32,\ \\
153\u{6}42,\ \u{5}13\u{6}42,\ 154\u{6}32,\ \u{5}14\u{6}32,\ 15\u{6}342,\ \u{5}1\u{6}342 
\end{multline*}
\begin{proposition}
    \label{prop:GamG_on_Fbasis}
    For any set composition $\II$, one ha:
    \[\GammaNC(G_\II) = \sum_{(\sigma,D) \in MP(\II)} \F_{(\sigma,D)}.\]
\end{proposition}
\begin{proof}
    Let $f$ be a $G_\II$ non-decreasing function from $[n]$ to $\N$.
    For each part $I_m$ in the set composition $\II$, let us consider the restriction $f_m$ of $f$ to $I_m$.
    Then there exists a unique word $w_m$ containing exactly once each number in $I_m$ such that
    \[ y\text{ appear before }z\text{ in }w_m \Leftrightarrow \begin{cases}
        f_m(y) \le f_m(z) &\text{ if }y < z;\\
        f_m(y) < f_m(z) &\text{ if }y > z;
    \end{cases}\]
    Indeed this word is obtained by ordering lexicographically the pair $((f_m(y),y))_{y \in I_m}$
    and keeping only the second element of each pair\footnote{
    Existence and uniqueness of the word $w_m$ can also be seen as a special case of Stanley fundamental
    theorem on $P$-partitions \cite[Theorem 6.2]{StanleyOrderedStructures}
    (see also \cite{KnuthPPartitions}), where the poset $P$
   has element set $I_m$ and no relations.}.

    We mark the descent in $w_m$ and by concatenating all the words $w_m$ (for $1 \le m \le r$),
    we get a descent-starred permutation $(\sigma,D)$ in $\MP(\II)$.
    This descent-starred permutation is the only one in $\MP(\II)$ such that $a_{f(1)} \cdots a_{f(n)}$
    appears in $\F_{(\sigma,D)}$, which explains the formula of the proposition.
\end{proof}

\begin{example}
    Take $\II_\ex$ as above, $\GammaNC(G_{\II_\ex})$ is given by \cref{eq:GamGIex}.
    The summation set can be split as follows:
    \begin{itemize}
        \item either $k_1 \le k_5$ or $k_5 < k_1$;
        \item besides, the integers $k_3$, $k_4$ and $k_6$ fulfill exactly one of the 6 following inequalities:
            \[ k_3 \le k_4 \le k_6, \ \ \ k_4 < k_3 \le k_6, \ \ \ k_3 \le k_6 < k_4, \]
            \[k_4 \le k_6 < k_3, \ \ \ k_6 < k_3 \le k_4, \ \ \ k_6 < k_4 < k_3.\]
    \end{itemize}
    Combining both case distinctions yield 12 different cases,
    and $\GammaNC(G_{\II_\ex})$ is a sum of 12 different terms which are the $\F$ functions
    indexed by the 12 descent-starred permutations in $\MP(\II)$ (which are listed above).
\end{example}

    \begin{corollary}
    The family $\big(\GammaNC(G_\II)\big)$ is a $\Z$-basis of $\WQSym$.
    \label{corol:GamG_Zbasis}   
    \end{corollary}
    \begin{proof}
        If $(\sigma,D)$ is the descent-starred permutation associated by \cref{lem:Bij_MP} 
        to a set composition $\II$ of $n$ of length $r$,
then the size of $D$ is $n-r$.
Besides, for each element $(\sigma',D') \in MP(\II)$, the size of $D'$ is smaller than $n-r$, unless
$(\sigma',D')=(\sigma,D)$.
Hence the proposition implies that the matrix of $\GammaNC(G_\II)$ in the basis $\F_{(\sigma,D)}$ is 
unitriangular with respect to the order
\[(\sigma',D') <_2 (\sigma,D) \Leftrightarrow |D'| < |D| \]
and $\GammaNC(G_\II)$ is a $\Z$-basis of $\WQSym$.
\end{proof}

\begin{remark}
    Stanley fundamental
    theorem on $P$-partitions \cite[Theorem 6.2]{StanleyOrderedStructures} 
    (see also Knuth's paper \cite{KnuthPPartitions}) implies
        that, if $G$ is a naturally labeled graph ({\em i.e.} such that $(i,j) \in E$ implies
        $i \le j$ as positive integer as positive integerss),
        then $\GammaNC(G)$ has a non-negative expansion on the $\F_{(\sigma,D)}$ basis.
        \cref{prop:GamG_on_Fbasis} gives examples of non-necessarily naturally labeled graphs $G$,
        such that the $\F_{(\sigma,D)}$ expansion of $\GammaNC(G)$ has non-negative coefficients.
        But, this is not the case for any graph $G$, as shown by the following example 
        (we skip details in the computation):
        \begin{multline*}
            \GammaNC\left(
\begin{array}{c}\begin{tikzpicture}[font=\scriptsize,scale=.3]
        \tikzstyle{vertex}=[circle,draw=black,inner sep=.3pt];
        \node[vertex] (v3) at (0,0) {$3$};
        \node[vertex] (v1) at (0,2) {$1$};
        \node[vertex] (v4) at (2,1) {$2$};
        \draw[->] (v3) -- (v1);
    \end{tikzpicture}\end{array}
    \right)=\F_{231}+\F_{\u{3}\u{2}1}+\F_{312}-\L_{\u{3}\u{2}1} \\
    =\F_{231}+\F_{312}+\F_{\u{3}21}+\F_{3\u{2}1}-\F_{321}.
\end{multline*}
    Such negative signs do not occur in the commutative setting:
    indeed, any function $\Gamma(\ov{G})$ is a non-negative linear
    combination of fundamental quasi-symmetric functions, 
    see \cite[Corollary 7.19.5]{Stanley:EC2}.
\end{remark}

\subsection{A generating family for the quotient}
\label{subsec:generating_quotient}

We will now show that $(G_\II)$, where $\II$ runs over all set compositions,
 is a generating family in the quotient $\G / \ideal$.
 As explained in \cref{sec:First_Main_Result}, 
 together with the results of \cref{sec:Zbasis_NonRes,sec:CIE_Rel},
this implies that $\GammaNC:\G / \ideal \to \WQSym$ is an isomorphism.

Here is the key combinatorial lemma in this section.
\begin{lemma}
    Let $G$ be a unlabeled poset. 
    Then either $G$ is equal to some $G_\II$ or, in the quotient $\G / \ideal$,
    one can write $G$ as a linear combination
    of graphs with the same set of vertices and more edges.
    \label{lem:IncreasingSize}
\end{lemma}

\begin{proof}
    Let $G$ be an acyclic directed graph with vertex set $[n]$ and edge set $E_G$.
    
    Throughout the proof, we denote $\sim$ the following symmetric relation:
    $x \sim y$ if, in $G$, there is no directed path (see \cref{Footnote:path} for the definition)
    from $x$ to $y$,
     nor from $y$ to $x$.
     When $x \sim y$, the graphs $G_{(x,y)}$ and $G_{(y,x)}$ obtained
          from $G$ by adding respectively an edge from $x$ to $y$ or
               from $y$ to $x$ are still acyclic.\bigskip

    We distinguish three cases.\medskip

    \noindent {\em Case 1: $G$ is not the graph of a transitive relation.}

    In other terms, there exist $x$, $y$ and $z$ such that
    \begin{itemize}
     \item there is an edge from $x$ to $y$ and from $y$ to $z$ in $G$;
     \item there is no edge from $x$ to $z$.
    \end{itemize}
    We consider $G_0=G_{(x,z)}$ the graph obtained from $G$ by adding
    an edge between $x$ and $z$.
    As a directed graph, $G_0$ is acyclic: otherwise, there would be a path from $z$ to $x$ in $G$
    and, together with $(x,y)$ and $(y,z)$, this path would be a directed cycle in $G$.
    But the non-oriented version of $G_0$ contains a cycle 
    $C = ( x,z,y)$.
    Using the notation of \cref{sub:DefCIE} (see also \cref{Special_Case:Transitivity}),
    one has $C^+ = \{ (x,z) \}$
    and the corresponding cyclic inclusion-exclusion element is
    \[ \CIE_{G_0,C} = G_0 - G.\] 
    Hence, in $\G/ \ideal$, one has $G=G_0$ and the statement is true in this case.\medskip
    
    This case is illustrated in \cref{fig:case1} with examples of graphs $G$ and $G_0$.
    Dashed edges are edges of $G$ and $G_0$ that do not play a role in the proof.\medskip

    \begin{figure}[t]
        \[
        G=\begin{array}{c} \begin{tikzpicture}[scale=.5,font=\scriptsize]
              \tikzstyle{vertex}=[circle,draw=black,inner sep=.3pt,minimum size=2mm];
             \node[vertex] (x) at (0,0) {$x$};
             \node[vertex] (y) at (-1,1.5) {$y$};
             \node[vertex] (z) at (0,3) {$z$};
             \node[vertex] (inutile1) at (2,0) { };
             \node[vertex] (inutile2) at (2,3) { };
             \draw[->] (x) -- (y);
             \draw[->] (y) -- (z);
%             \draw[->] (x) -- (z);
             \draw[->,dashed] (x) -- (inutile2);
             \draw[->,dashed] (inutile1) -- (inutile2);
             \draw[->,dashed] (inutile1) -- (z);
         \end{tikzpicture}\end{array}
         \qquad \qquad
        G_0=\begin{array}{c} \begin{tikzpicture}[scale=.5,font=\scriptsize]
              \tikzstyle{vertex}=[circle,draw=black,inner sep=.3pt,minimum size=2mm];
             \node[vertex] (x) at (0,0) {$x$};
             \node[vertex] (y) at (-1,1.5) {$y$};
             \node[vertex] (z) at (0,3) {$z$};
             \node[vertex] (inutile1) at (2,0) { };
             \node[vertex] (inutile2) at (2,3) { };
             \draw[->] (x) -- (y);
             \draw[->] (y) -- (z);
             \draw[->] (x) -- (z);
             \draw[->,dashed] (x) -- (inutile2);
             \draw[->,dashed] (inutile1) -- (inutile2);
             \draw[->,dashed] (inutile1) -- (z);
         \end{tikzpicture}\end{array}
         \]
        \caption{Graphs $G$ and $G_0$ in the first case of the proof of \cref{lem:IncreasingSize}.}
        \label{fig:case1}
    \end{figure}
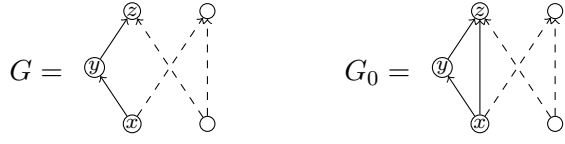

    \noindent {\em Case 2: the relation $\sim$ is not an equivalence relation.}

    By assumption, there exist vertices $x,y,z$ such that
    \begin{itemize}
     \item there is a path $(x, v_1, \cdots, v_k, z)$ from $x$ to $z$ in $G$;
     \item one has $x \sim y$ and $y \sim z$.
    \end{itemize}
    By definition of $\sim$, the graph $G_{(x,y)}$ is acyclic.
    Moreover, it does not contain a path from $z$ to $y$.
    Indeed, as $y \sim z$ in $G$, such a path should use the edge $(x,y)$
    and thus be the concatenation of a path from $z$ to $x$ with the edge $(x,y)$.
    But $G$ does not contain a path from $z$ to $x$ (indeed, it contains a path from $x$ to $z$
    and no directed cycles).
    
    Therefore, the graph $G_0$ obtained from $G_{(x,y)}$ by adding an edge from $y$ to $z$
    is an acyclic directed graph.
    However, its undirected version contains a cycle 
    \[C = (x,y,z,v_k,\cdots,v_1).\]
    Using the notation of \cref{sub:DefCIE}, for this cycle,
    one has $C^+ = \{ (x,y),(y,z) \}$.
    Hence,
    \[ \CIE_{G_0,C} = G_0 - G_0 \setminus \{(x,y)\} - G_0 \setminus \{(y,z)\} + G_0 \setminus \{(x,y),(y,z)\}.\]
    But $G_0 \setminus \{(x,y),(y,z)\}$ is $G$, so, in the quotient $\G / \ideal$, one has
    \[ G = - G_0 + G_0 \setminus \{(x,y)\} + G_0 \setminus \{(y,z)\} \]
    and the statement is proved in this case.\medskip

    This case is illustrated in \cref{fig:case2} with examples of graphs $G$ and $G_0$.
    Here, the dashed edge illustrates the fact that we do not know the length of the path $P$ from $x$ to $z$.
    Potential extra edges and vertices of $G$ and $G_0$ have not been represented for more readability.\medskip

    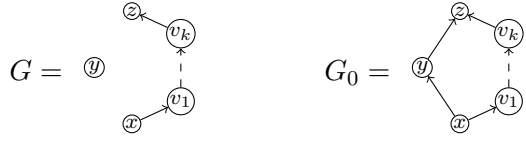
\begin{figure}[t]
        \[
        G=\begin{array}{c} \begin{tikzpicture}[scale=.5,font=\scriptsize]
              \tikzstyle{vertex}=[circle,draw=black,inner sep=.3pt,minimum size=2mm];
             \node[vertex] (x) at (0,0) {$x$};
             \node[vertex] (y) at (-1,1.5) {$y$};
             \node[vertex] (z) at (0,3) {$z$};
             \node[vertex] (chaine1) at (1.3,.6) { ${\tiny v_1}$};
             \node[vertex] (chaine2) at (1.3,2.4) { ${\tiny v_k}$};
%             \draw[->] (x) -- (y);
%             \draw[->] (y) -- (z);
             \draw[->] (x) -- (chaine1);
             \draw[->,dashed] (chaine1) -- (chaine2);
             \draw[->] (chaine2) -- (z);
         \end{tikzpicture}\end{array}
         \qquad \qquad
        G_0=\begin{array}{c} \begin{tikzpicture}[scale=.5,font=\scriptsize]
              \tikzstyle{vertex}=[circle,draw=black,inner sep=.3pt,minimum size=2mm];
             \node[vertex] (x) at (0,0) {$x$};
             \node[vertex] (y) at (-1,1.5) {$y$};
             \node[vertex] (z) at (0,3) {$z$};
             \node[vertex] (chaine1) at (1.3,.6) { ${\tiny v_1}$};
             \node[vertex] (chaine2) at (1.3,2.4) { ${\tiny v_k}$};
             \draw[->] (x) -- (y);
             \draw[->] (y) -- (z);
             \draw[->] (x) -- (chaine1);
             \draw[->,dashed] (chaine1) -- (chaine2);
             \draw[->] (chaine2) -- (z);
         \end{tikzpicture}\end{array}
         \]
        \caption{Graphs $G$ and $G_0$ in the second case of the proof of \cref{lem:IncreasingSize}.}
        \label{fig:case2}
    \end{figure}

    \noindent {\em Case 3: $G$ is the graph of a transitive relation and 
    the relation $\sim$ is an equivalence relation.}

    In this case, we will prove that $G$ is necessarily equal to $G_\II$,
    for some set composition $\II$.

    Let us start by a remark: in the graph of a transitive relation,
    the existence of a path from $x$ to $y$ implies the existence of an edge from
    $x$ to $y$.
    Hence $x \nsim y$ means that there is either an edge from $x$ to $y$ or from $y$ to $x$.

    Denote $(V_j)_{j \in J}$ the partition of the vertex set of $G$ into equivalence classes of $\sim$.
    Consider two such classes $V_j$ and $V_k$.
    We will prove that either $V_j \times V_k$ or $V_k \times V_j$ is included in $E_G$.

    Select arbitrarily a pair $(v_0,w_0)$ in $V_j \times V_k$.
    As $v_0 \nsim w_0$, by eventually swapping $v_0$ and $w_0$ (and simultaneously $j$ and $k$), 
    we may assume that $(v_0,w_0)$ is an edge of $G$.
 
    Then, for any $w$ in $V_k$, the pair $(v_0,w)$ is also an edge of $G$.
    Indeed, if this is not the case, as $v_0 \nsim w$, this would imply that $(w,v_0)$ is an edge of $V$.
    But, then by transitivity, $(w,w_0)$ should be an edge of $G$, which is impossible as $w \sim w_0$.

    The same argument proves that, for any $v$ in $V_j$, the pair $(v,w)$ must be an edge of $G$,
    which proves the inclusion of $V_j \times V_k$ in $E_G$.

    As we may have swapped $v_0$ and $w_0$ at the beginning,
    we have in fact proved that for any pair $(j,k)$ in $J^2$,
    either $V_j \times V_k$ or $V_k \times V_j$ is included in $E_G$.
    As $G$ does not have any directed cycle, there exists a total order $<_J$ on $J$
    such that $V_j \times V_k$ is included in $E_G$ if and only if $j <_J k$.
    
    By definition of $\sim$, there is no edges with both extremities in the same $V_j$.
    Besides, there can not be an edge from $V_k$ to $V_j$ (with $j<_J k$),
    as this would create a directed cycle of length $2$.
    Finally, the set of edges of $G$ is exactly
\[ \bigsqcup_{j <_J k} V_j \times V_k, \]
which means that $G=G_\II$ for $\II=(V_j)_{j \in J}$.
\end{proof}

Let $G$ be an acyclic directed graph.
Iterating \cref{lem:IncreasingSize}, one can write $G$ as an integer linear combination
of $G_\II$ in the quotient space $\G / \ideal$.
In other terms, $G_\II$ is a generating family of the vector space $\G / \ideal$.
\label{GI_Span_Quotient}

\subsection{First main result}
\label{sec:First_Main_Result}
We are now ready to prove the following statement.
\begin{theorem}
The space $\ideal$, spanned by cyclic inclusion-exclusion elements,
is the kernel of the surjective morphism $\GammaNC$
from $\G$ to $\WQSym$.
    \label{thm:NonCom_NonRestr}
\end{theorem}
\begin{proof}
    Denote $\Ker$ the kernel of $\GammaNC$. 
    By \cref{prop:CIE_in_Kernel}, it contains $\ideal$.
    On the one hand (\cref{GI_Span_Quotient}), we know that $\G / \ideal$
    is spanned by the family $(G_\II)$.
    On the other hand (\cref{corol:GamG_Zbasis}),
    the family $\GammaNC(G_\II)$ is a basis of $\WQSym$,
    which implies in particular that the $(G_\II)$ are linearly independent
    in $\G/\Ker$ and hence in $\G / \ideal$.

    Therefore $(G_\II)$ is a basis of $\G / \ideal$
    and $\GammaNC$ is an isomorphism from $\G / \ideal$ to $\WQSym$ 
    (it sends a basis on a basis), which concludes the proof.
\end{proof}

\begin{remark}
    In fact, we have proved a stronger result:
    the subspace of $\G$ spanned by cyclic inclusion-exclusion associated to
    cycles $C$ with $|C^+|=1$ and $|C^+|=2$ is the kernel of $\Gamma$
    (and hence coincides with $\ideal$).
\end{remark}
\subsection{Unlabeled commutative framework and second main result}
Consider a unlabeled directed graph $\ov{G}$ and a cycle
$\ov{C}$ of the undirected version of $\ov{G}$.
As in \cref{sub:DefCIE}, we can define $\ov{C}^+$ and an element 
\[\ov{\CIE}_{\ov{G},\ov{C}} = \sum_{D \subseteq \ov{C}^+} 
(-1)^{|D|} \ov{G} \setminus D.\]
But $\ov{G}$ is the equivalence class of some graph $G$,
whose undirected version contains a cycle $C$, which projects on $\ov{C}$.
With this in mind,
$\ov{\CIE}_{\ov{G},\ov{C}}$ 
is simply the image of $\CIE_{G,C}$ by the morphism $\phi_u:\G \to \ov{\G}$.

Let us consider the subspace $\ov{\ideal}$ of $\ov{\G}$ spanned 
by cyclic inclusion-exclusion elements.
Equivalently this is the image of $\ideal$ by the morphism $\phi_c$.

\begin{theorem}
    The ideal $\ov{\ideal}$, spanned by inclusion-exclusion elements,
is the kernel of the surjective morphism $\Gamma$
from $\ov{\G}$ to $\QSym$.
    \label{thm:Com_NonRestr}
\end{theorem}
\begin{proof}
    This follows from \cref{thm:NonCom_NonRestr}, and the fact that the morphism
    $\GammaNC$ is compatible with the action of $S_n$ on homogeneous components
    described in \cref{sec:WQSym,sec:graphs}.
    Indeed, one can write
\begin{multline*} \ov{\G} / \langle \ov{\CIE}_{\ov{G},\ov{C}} \rangle 
\simeq \big( \G / \langle x - \sigma.x \rangle \big) / \langle \ov{\CIE_{G,C}} \rangle
\simeq \G / \langle x - \sigma.x, \CIE_{G,C} \rangle \\
\simeq   \big(\G / \langle \CIE_{G,C} \rangle \big) / \langle x - \sigma.x \rangle
\simeq  \WQSym / \langle x - \sigma.x \rangle
\simeq \QSym. \qedhere
\end{multline*}
\end{proof}

\begin{remark}
    The function $\Gamma\big(\ov{G_\II}\big)$ in $\QSym$ depends only 
    on the integer composition $I=\phi_c(\II)$.
    Therefore, from \cref{sec:Zbasis_NonRes}, we know that this family, indexed by integer compositions,
    is a $\Z$-basis of $\QSym$.
    This family has appeared in a paper of Stanley \cite[Note p7]{StanleyDescentConnectivitySets}
    which noticed that the change of basis matrix with the fundamental basis is unitriangular
    (commutative version of \cref{prop:GamG_on_Fbasis}).
\end{remark}

\begin{remark}
    A direct proof of \cref{thm:Com_NonRestr} along the same lines 
    as the proof of \cref{thm:NonCom_NonRestr} is of course possible.
\end{remark}

\section{The kernel in the bipartite case}
\label{sec:Bip}

The purpose of this Section is to show that the kernel of
$\Gamma$ and $\GammaNC$ restricted to bipartite graphs
is also generated by cyclic-inclusion relations.

\subsection{Preliminaries for the bipartite setting}
Recall that a directed graph is called {\em bipartite} if its vertex set can be split in $V \sqcup W$,
such that if $(v,w) \in E$, then $v$ lies in $V$ and $w$ in $W$ (in other words,
the edge set is included in $V \times W$).
Note that this bipartition is not unique as isolated vertices can be either in $V$ or $W$,
but this is the only degree of freedom.

The subalgebra of the graph algebra $\G$ spanned by bipartite graphs
will be denoted $\BG$.
If $B$ is a bipartite graph and $C$ a cycle in the undirected version of $B$,
then the cyclic inclusion-exclusion element $\CIE_{B,C}$ lies in $\BG$.
We denote $\Bideal$ the subspace of $\BG$ spanned by these elements.

Finally, we consider the restriction of $\GammaNC$ to $\BG$, that we denote $\BGammaNC$.
Clearly, from \cref{prop:CIE_in_Kernel}, the space $\Bideal$ is included in the kernel of $\BGammaNC$.

\begin{remark}
    \label{rem:why_non_trivial}
    The kernel of $\BGammaNC$ is, from \cref{thm:NonCom_NonRestr}, equal to
    $\ideal \cap \BG$.
    But, even if $\ideal$ is by definition generated by cyclic inclusion-exclusion elements,
    we do not know {\em a priori} whether this intersection is spanned by the cyclic inclusion-exclusion
    elements that lie in it.
\end{remark}

\subsection{The bipartite graphs $B_{(\II,\JJ)}$}
Consider a set composition of $[n]$.
In the following, it will be convenient to distinguish odd and even-indexed parts
of the composition.
Therefore we denote $I_1$ its first part, $J_1$ its second part, $I_2$ its third and so on
until $J_r$ which is eventually empty if the number of parts of the set composition
is odd.
In this context, a set composition is denoted $(\II,\JJ)$
and $r$ is called its {\em semi-length}.
We draw the attention of the reader on the fact that, from this viewpoint,
a pair $(\II,\JJ)$ is a {\em single} set composition and not a pair
of set compositions.
\begin{definition}
Let $(\II,\JJ)$ be a set composition of $[n]$.
We consider the bipartite directed graph $B_{(\II,\JJ)}$ with vertex set $[n]$
and edge set
\[ \bigsqcup_{h < k} I_h \times J_k. \]
\end{definition}

\begin{example}
    \label{ex:BIJ}
    Consider the set composition $26|4|5|17|3$.
    With the notations of this section, it writes as $(\II_\ex,\JJ_\ex)=(26|5|3,4|17|)$ 
    (in this case , $J_3$ is empty, which explains the vertical bar at the end
    of the numerical notation of $\JJ_\ex$).
    The graph 
    $B_{(\II_\ex,\JJ_\ex)}$ and the associated word quasi symmetric function are
     \begin{equation}
        \!\!  B_{(\II_\ex,\JJ_\ex)}=\begin{array}{c}
         \begin{tikzpicture}[font=\tiny,scale=.27]
        \tikzstyle{vertex}=[circle,draw=black,inner sep=.3pt];
        \node[vertex] (v2) at (0,0) {$2$};
        \node[vertex] (v6) at (1.5,0) {$6$};
        \node[vertex] (v4) at (.75,2.5) {$4$};
        \node[vertex] (v5) at (4.5,0) {$5$};
        \node[vertex] (v1) at (3.75,2.5) {$1$};
        \node[vertex] (v7) at (5.25,2.5) {$7$};
        \node[vertex] (v3) at (7.5,0) {$3$};
        \draw[->] (v2) -- (v1);
        \draw[->] (v2) -- (v4);
        \draw[->] (v2) -- (v7);
        \draw[->] (v6) -- (v1);
        \draw[->] (v6) -- (v4);
        \draw[->] (v6) -- (v7);
        \draw[->] (v5) -- (v1);
        \draw[->] (v5) -- (v7);
        \end{tikzpicture}
     \end{array}; \quad
     \GammaNC(B_{(\II_\ex,\JJ_\ex)}) = \!\!\! \sum_{ \gf{k_1,\dots,k_7}{
    \gf {\max(k_2,k_6) \le \min(k_1,k_4,k_7) }{k_5 \le \min(k_1,k_7)}}}
    \!\!\! a_{k_1} \cdots a_{k_7}.
    \label{eq:GamBIJex}
\end{equation}
\end{example}

\subsection{A combinatorial lemma}
If $V \sqcup W = [n]$ is a bipartition of $[n]$,
we denote $K_{V,W}$ the complete directed bipartite graph between $V$
and $W$, that is the graph with vertex set $[n]$ and edge set $V \times W$.
Let $D$ be a subset of $V \times W$.
Then we consider the directed graph $K^D$ obtained from $K_{V,W}$ by 
turning the edges in $D$ around (in general, $K^D$ is not a directed bipartite graph).

For example, consider $V=\{1,2,4,6\}$ and $W=\{3,5\}$.
The corresponding complete bipartite graph is the left-most graph in \cref{fig:Turning_Edges_Around}.
We now choose a subset of $V \times W$, {\em e.g.} $D=\{(2,3),(6,3)\}$.
The corresponding graph $K^D$ is drawn
in the middle of \cref{fig:Turning_Edges_Around}.\medskip

\begin{figure}[t]
    \[
    \begin{array}{c}   \begin{tikzpicture}[font=\scriptsize,scale=.32]
        \tikzstyle{vertex}=[circle,draw=black,inner sep=.3pt];
        \node[vertex] (v1) at (0,0) {$1$};
        \node[vertex] (v2) at (1.5,0) {$2$};
        \node[vertex] (v4) at (3,0) {$4$};
        \node[vertex] (v6) at (4.5,0) {$6$};
        \node[vertex] (v3) at (1.25,2.5) {$3$};
        \node[vertex] (v5) at (3.25,2.5) {$5$};
        \draw[->] (v1) -- (v3);
        \draw[->] (v1) -- (v5);
        \draw[->] (v2) -- (v3);
        \draw[->] (v2) -- (v5);
        \draw[->] (v4) -- (v3);
        \draw[->] (v4) -- (v5);
        \draw[->] (v6) -- (v3);
        \draw[->] (v6) -- (v5);
    \end{tikzpicture} \end{array} 
    \qquad
    \begin{array}{c}   \begin{tikzpicture}[font=\scriptsize,scale=.32]
        \tikzstyle{vertex}=[circle,draw=black,inner sep=.3pt];
        \node[vertex] (v1) at (0,0) {$1$};
        \node[vertex] (v2) at (1.5,0) {$2$};
        \node[vertex] (v4) at (3,0) {$4$};
        \node[vertex] (v6) at (4.5,0) {$6$};
        \node[vertex] (v3) at (1.25,2.5) {$3$};
        \node[vertex] (v5) at (3.25,2.5) {$5$};
        \draw[->] (v1) -- (v3);
        \draw[->] (v1) -- (v5);
        \draw[very thick,->] (v3) -- (v2);
        \draw[->] (v2) -- (v5);
        \draw[->] (v4) -- (v3);
        \draw[->] (v4) -- (v5);
        \draw[very thick,->] (v3) -- (v6);
        \draw[->] (v6) -- (v5);
    \end{tikzpicture} \end{array} 
    \qquad
    \begin{array}{c}   \begin{tikzpicture}[font=\scriptsize,scale=.27]
        \tikzstyle{vertex}=[circle,draw=black,inner sep=.3pt];
        \node[vertex] (v1) at (0,0) {$1$};
        \node[vertex] (v2) at (0,3) {$2$};
        \node[vertex] (v4) at (3,0) {$4$};
        \node[vertex] (v6) at (3,3) {$6$};
        \node[vertex] (v3) at (1.5,1.5) {$3$};
        \node[vertex] (v5) at (1.5,4.5) {$5$};
        \draw[->] (v1) -- (v3);
        \draw[->] (v1) to [bend left=78] (v5);
        \draw[->] (v3) -- (v2);
        \draw[->] (v2) -- (v5);
        \draw[->] (v4) -- (v3);
        \draw[->] (v4) to [bend right=78] (v5);
        \draw[->] (v3) -- (v6);
        \draw[->] (v6) -- (v5);
    \end{tikzpicture} \end{array} 
    \]
    \caption{A complete bipartite graph (left-most graph), 
    the graph obtained after turning some edges around (middle graph)
    and a graph from the family $H_{(\II,\JJ)}$ (right-most graph).
    Note that the last two are identical.}
    \label{fig:Turning_Edges_Around}
\end{figure}
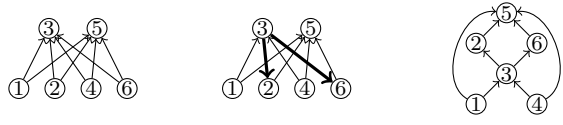

We are also interested in the following family of graphs.
If $(\II,\JJ)$ is a set composition of $[n]$, we define $H_{(\II,\JJ)}$ as the graph
with vertex set $[n]$ and edge set
\[ \bigsqcup_{m \le m'} (I_m \times J_{m'}) \sqcup 
\bigsqcup_{m < m'} (J_m \times I_{m'}). \]
%%%%%% INUTILE %%%%%%%%%%%%%%%%%%%%
%Note that $H_{(\II,\JJ)}$ is close to the graph $G_\KK$ from \cref{sec:GI},
%where $\KK=(I_1,J_1,I_2,\cdots)$,
%except that we only have edges between levels with different parity.
%%%%%%%%%%%%%%%%%%%%%%%%%%%%%%%%%%%

As an example, let us choose $\II=14|26$ and $\JJ=3|5$, that is $\KK=14|3|26|5$.
The corresponding graph $H_{(\II,\JJ)}$ is the right-most graph of
\cref{fig:Turning_Edges_Around}.
The examples have been chosen so that $H_{(\II,\JJ)}$ and $K^D$ are the same graph.
We will now see that the family $H_{(\II,\JJ)}$ roughly corresponds
to the family of {\em acyclic} graphs among the $K^D$.

The following lemma will be useful in the next Section.
\begin{lemma}
    Let $V$, $W$ and $D$ as above.
    Assume that each vertex in $W$ is the extremity of at 
    least one edge not in $D$.
    Then, either $K^D$ contains a directed cycle,
    or there exists a set composition $(\II,\JJ)$
    with $\bigsqcup_{1 \le k \le r}  I_k = V$ and $\bigsqcup_{1\le k \le r} J_k = W$ 
    such that $K^D=H_{(\II,\JJ)}$.

    Moreover, each such set composition $(\II,\JJ)$ corresponds to exactly
    one set $D$ such that $K^D$ is acyclic. 
    \label{lem:Reversing_Edges}
\end{lemma}
\begin{proof}
    Assume $K^D$ is acyclic.
    Denote $I_1$ the subset of $V$ of elements $x$ such that 
    \[\{(x,y), y \in W\} \cap D =\emptyset,\]
    {\it i.e.} none of the edges starting $x$ have been turned around.

    We will prove by contradiction that $I_1$ is non empty.
    Assume $I_1=\emptyset$. 
    Then $K^D$ 
    is a directed graph, where all vertices have at least one incoming edge
    (vertices in $W$ have at least one incoming edge because of our hypothesis
    and vertices in $V$ have an incoming edge because $I_1$ is empty). 
    Such a graph necessarily contains a directed cycle (start from an arbitrary vertex
    and follow backwards incoming edges until you encounter twice the same vertex, 
    which will happen eventually; you have found a directed cycle).

    Thus $I_1$ is non-empty and,
    by construction, $I_1 \times W$ is included
    in the edge set $E(K^D)$ of $K^D$.

    Consider now the set $J_1$ of elements $y$ such that
    \[\{(x,y), x \in V \setminus I_1\} \subseteq D,\]
    {\it i.e.} all edges going to $y$, except those starting from an element of $I_1$,
    have been turned around.
    
    We will prove by contradiction that $J_1$ is non empty.
    Assume $J_1=\emptyset$. 
    Then the graph induced by $K^D$ on the set $[n] \setminus I_1$
    is a directed graph, where all vertices have at least one incoming edge
    (vertices in $W$ have at least one incoming edge in this induced graph
    because we have assumed $J_1$ empty
    and vertices in $V \setminus I_1$ have an incoming edge because they do not
    belong to $I_1$). 
    This graph should contain a directed cycle and we reach a contradiction.

    Thus $J_1$ is non-empty and, by construction,
    $J_1 \times (V \setminus I_1)$ is included in $E(K^D)$.

    Consider now the subset $I_2$ of $V \setminus I_1$ of elements $x$ such that     
    \[\{(x,y), y \in (W \setminus J_1) \} \cap D =\emptyset.\]
    The same proof as above (considering the graph induced on $[n] \setminus (I_1 \cup J_1)$)
    shows that, if $I_1 \subsetneq V$, then $I_2$ is non-empty.
    By construction, $I_2 \times (W \setminus J_1)$ is included in $E(K^D)$.

    We keep going like this, defining, for each $m \ge 1$,
    \begin{align*}
        I_m &= \bigg\{ x \in V_{m-1} \text{ s.t. }
        \{(x,y), y \in W_{m-1}  \cap D\} =\emptyset \bigg\} ;\\
        J_m &= \bigg\{ y \in W_{m-1}  \text{ s.t. }
     \{(x,y), x \in V_m \} \subseteq D \bigg\},
    \end{align*}
    where we set $V_m=V  \setminus (I_1 \cup \dots \cup I_{m})$ and 
    $W_m=W  \setminus (J_1 \cup \dots \cup J_{m})$.
    We stop the construction when $I_1 \sqcup  \cdots \sqcup I_r =V$,
    which automatically implies $J_1 \sqcup  \cdots \sqcup J_r =W$.
    Then the argument above shows that
    all sets $I_m$ and $J_m$, except possibly $J_r$, are non-empty
    (which explains that the construction above always ends)
    and, by construction, if $1\le m\le r$,
    \begin{align*}
        I_m \times W_{m-1} \subseteq E(K^D); \\
        J_m \times V_{m} \subseteq E(K^D).
    \end{align*}
    In other terms, the edge set of $K^D$ contains the one of $H_{(\II,\JJ)}$.
    But for all $(v,w)$ in $V \times W$, either $(v,w)$ or $(w,v)$ is an edge
    of $H_{(\II,\JJ)}$, so that $K^D$ cannot have more edges.
    Thus $K^D=H_{(\II,\JJ)}$, as wanted.

    The fact that each set composition
    with $\bigsqcup_{1 \le k \le r}  I_k = V$ and $\bigsqcup_{1\le k \le r} J_k = W$ 
    corresponds to exactly one set $D$ is trivial:
    just take $D$ as the set of edges which are oriented from $W$ to $V$
    in the graph $H_{(\II,\JJ)}$.
\end{proof}

\subsection{Another $\Z$-basis of $\WQSym$}
\label{sec:Zbasis_Res}
The purpose of this section is to prove that
the word quasi-symmetric functions $\GammaNC(B_{(\II,\JJ)})$
form a $\Z$-basis of $\WQSym$, when $(\II,\JJ)$ runs over all set compositions.

As in \cref{sec:Zbasis_NonRes}, we use an intermediate family.
If $(\II,\JJ)$ is a set composition of $[n]$, define
\begin{equation}
    \NN_{(\II,\JJ)} = \sum a_{k_1} \cdots a_{k_n},
    \label{eq:def_NIJ}
\end{equation}
where the sum runs over lists $(k_1,\dots,k_n)$ that satisfy:
\begin{itemize}
    \item if $x$ is in $I_m$ and $y$ in $J_m$ for some index $m \le r$,
        then $k_x \le k_y$ ;
    \item if $x$ is in $J_m$ and $y$ in $I_{m+1}$ for some index $m \le r-1$,
        then $k_x < k_y$.
\end{itemize}
For example, continuing \cref{ex:BIJ}, one has:
\[\NN_{(\II_\ex,\JJ_\ex)}= \sum_{ \gf{k_1,\dots,k_7}{
\gf {\max(k_2,k_6) \le k_4 < k_5 \le \min(k_1,k_7)}{\max(k_1,k_7) < k_3}} }
    \!\!\! a_{k_1} \cdots a_{k_7}.\]
In general, denote $m(x)$ the index $m$ such that $x$ lies in $I_m$ or $J_m$.
Then the inequalities above on the indices $k_x$ automatically
imply that $k_x < k_y$ whenever $m(x) < m(y)$.
\begin{remark}
    The family $(\NN_{(\II,\JJ)})$ has been recently considered by the author
    and coauthors in~\cite{AvecLesJeansFPSAC}
    (our family $\NN_{(\II,\JJ)}$ corresponds
    to $\FF(\PP_\KK)$ with the notations of~\cite{AvecLesJeansFPSAC}).
    The commutative projection of $\NN_{(\II,\JJ)}$ had appeared before:
    indeed, it coincides with a $\Z$-basis of $\WQSym$
    introduced by K. Luoto in \cite{LuotoBasisQSym}
    (denoted $N$ in Luoto's paper).
\end{remark}
\begin{proposition}
    The family $(\NN_{(\II,\JJ)})$, where $(\II,\JJ)$ runs over all set compositions,
    is a $\Z$-basis of $\WQSym$.
\end{proposition}
\begin{proof}
    See \cite[Proposition 5.4]{AvecLesJeansFPSAC}.
\end{proof}
\begin{remark}
    A surprising fact in this proof is that we have not been able
    to find some other $\Z$-basis of $\WQSym$ with a unitriangular 
    change-of-basis matrix.
    The proof uses an evaluation on a virtual alphabet
    which turns $(\NN_{(\II,\JJ)})$ into a two-alphabet version,
    whose linear independence is easy to observe.
    
    Such a trick is not needed in the commutative setting --
    see \cite[proof of Theorem 3.4]{LuotoBasisQSym}.
    Finding a more elementary proof in the noncommutative setting
    would certainly be interesting.
\end{remark}

A nice feature of this basis is that,
for any bipartite graph $B$, 
the associated word quasi-symmetric function $\GammaNC(B)$
can be written as a multiplicity-free sum of $\NN$ function.
A weaker version of the following proposition was announced 
in \cite{AvecLesJeansFPSAC} (see Proposition 5.5 there).
\begin{proposition}
    \label{prop:Mult_Free_Sum}
    Let $B$ be a bipartite graph with vertex set $[n]$ and edge set $E_B$
    and consider {\em the} bipartition $[n]=V \sqcup W$ of its vertex set 
    so that $E_B \subseteq V \times W$ and $W$ contains no isolated vertex.
    Then
    \[ \GammaNC(B) = \sum \NN_{(\II,\JJ)},\]
    where the sum runs over set compositions $(\II,\JJ)$ such that:
    \begin{itemize}
        \item $\bigsqcup_{1 \le k \le r}  I_k = V$ and $\bigsqcup_{1\le k \le r} J_k = W$ ;
        \item $(x,y) \in E_B \implies m(x) \le m(y)$.
    \end{itemize}
\end{proposition}
\begin{proof}
    We denote $\ov{E_B}$ the set of {\em non-edges} of $B$,
    that is $(V \times W) \setminus E_B$.
    Consider a $B$ non-decreasing function $f: [n] \to \N$.
    For each non-edge $(x,y) \in \ov{E_B}$, one has either $f(x) \le f(y)$ or $f(y) < f(x)$.
    This trivial remark allows us to decompose
    \[\left\{ f : [n] \to \N, \text{ $f$ is $B$ non-decreasing}\right\} =
    \bigsqcup_{D \subseteq \ov{E_B}} \FFF_{D},\]
    where $\FFF_{D}$ is the set of $B$ non-decreasing functions that satisfy:
    \begin{itemize}
        \item $f(y) < f(x)$ for each $(x,y)$ in $D$;
        \item $f(x) \le f(y)$ for each $(x,y)$ in $\ov{E_B} \setminus D$.
    \end{itemize}
    This decomposition yields the formula
    \begin{equation}
        \GammaNC(B) = \sum_{D \subseteq \ov{E_B}} \NN_{D},
        \label{eq:NDec_GamB}
    \end{equation}
    where $\NN_{D}=\sum_{f \in \FFF_{D}} a_{f(1)} \cdots a_{f(n)}$.
    We will prove that, for each set $D$,
    the word quasi-symmetric function $\NN_{D}$
    is either $0$ or equal to one of the basis element $\NN_{(\II,\JJ)}$.\medskip

    Fix a subset $D$ of $\ov{E_B}$.
    Note that $\ov{E_B}$, and thus $D$, seen as a subset of $V\times W$,
    satisfies the hypothesis of \cref{lem:Reversing_Edges} as
    we assumed that $W$ contains no isolated vertex.
    Applying \cref{lem:Reversing_Edges}, we are left with two cases.
    \begin{itemize}
        \item Either the graph $K^D$ contains a directed cycle
          \[ (x_1,y_1,x_2,y_2, \cdots ,x_k,y_k),\]
          where $x_\ell$, respectively $y_\ell$, lies in $V$, respectively $W$ 
          (for $1 \le \ell \le k$). Then any function $f$ in $\FFF_{D}$ satisfies
          \[ f(x_1) \le f(y_1) < f(x_2) \le \cdots 
          < f(x_k) \le f(y_k) <f(x_1),\]
          which is clearly impossible.
          Thus $\FFF_{D}$ is empty and $\NN_{D}=0$.
      \item Or the graph $K^D$ is identical to some $H_{(\II,\JJ)}$ for some
          set composition $(\II,\JJ)$.
          In this case, functions $f$ in $\FFF_{D}$ fulfills by definition
          \[\begin{cases}
              f(x) \le f(y) &\text{ if } (x,y) \in (V \times W) \setminus D\text{, that is if }
              x \in I_m\text{ and } y \in J_{m'}\text{ with }m \le m';\\
              f(y) < f(x) &\text{ if } (x,y) \in D\text{, that is if }
              y \in J_m\text{ and } y \in I_{m'}\text{ with }m < m';\\
          \end{cases}\]
          These functions correspond to the lists $(k_1,\cdots,k_n)$
          in the summation index in the definition of $\NN_{(\II,\JJ)}$
          in \cref{eq:def_NIJ}.
          Therefore $\NN_{D}=\NN_{(\II,\JJ)}$.
    \end{itemize}

    It remains to prove that each set composition $(\II,\JJ)$ with the conditions
    given in the Proposition appears exactly once.
    This is a consequence of the second part of \cref{lem:Reversing_Edges}:
    there is a one-to-one correspondence between subset $D \subseteq V \times W$
    such that $K^D$ is acyclic and set compositions $(\II,\JJ)$
    with $\bigsqcup_{1 \le k \le r}  I_k = V$ and $\bigsqcup_{1\le k \le r} J_k = W$.
    In this correspondence, the fact that $D \subseteq \ov{E_B}$ translates as
        \[(x,y) \in E_B \implies m(x) \le m(y),\]
    which concludes the proof of the proposition.
\end{proof}
\begin{example}
    Consider the graph $B=B_{(\II_\ex,\JJ_\ex)}$ from \cref{ex:BIJ}.
    In this case $\ov{E_B}=\{(5,4),\ (3,4),\ (3,1),\ (3,7)\}$.
    It has 16 subsets $D$. Among these 16 sets $D$,
    exactly $3$ of them lead to a graph $K^D$ with a directed cycle:
    the one where $D$ contains $(5,4)$ but not $(3,4)$ and either $(3,1)$
    or $(3,7)$ or both.
    The other $13$ sets $D$ yield each a basis element $\NN_{(\II,\JJ)}$
    in the expansion of $\GammaNC(B_{(\II_\ex,\JJ_\ex)})$, which is:
    \begin{multline*}
        \GammaNC(B_{(\II_\ex,\JJ_\ex)})
        = \NN_{(26|5|3,4|17|)} 
        + \NN_{(26|5|3,4|1|7)} + \NN_{(26|5|3,4|7|1)} 
        +\NN_{(26|35,4|17)} \\ + \NN_{(236|5,4|17)}
        + \NN_{(256|3,147|)} 
        + \NN_{(256|3,14|7)}+ \NN_{(256|3,17|4)} \\ + \NN_{(256|3,47|1)}
        + \NN_{(256|3,1|47)}+ \NN_{(256|3,4|17)}+ \NN_{(256|3,7|14)}
        + \NN_{(2356,147)}
    \end{multline*}
    One can check that these $13$ set compositions are exactly the ones
    that fulfill the condition from \cref{prop:Mult_Free_Sum}.
\end{example}

\begin{corollary}
    \label{corol:ZBasis_Restr}
The family $(\GammaNC(B_{(\II,\JJ)}))$, when $(\II,\JJ)$ runs over all set compositions,
is a $\Z$-basis of $\WQSym$.    
\end{corollary}
\begin{proof}
We endow set compositions $(\II,\JJ)$ with the lexicographic containment order
on $(I_1,\ov{J_1},I_2,\ov{J_2},\dots)$ 
($\ov{J_m}$ denotes here the complement of $J_m$ in $W$) that is
\[(\II,\JJ) \preceq (\II',\JJ') \text{ if and only if }
\begin{cases}
    \phantom{\text{ or }(}I_1 \subsetneq I'_1 \\
    \text{ or }(I_1=I'_1 \text{ and } J_1 \supsetneq J'_1) \\
    \text{ or }(I_1=I'_1 \text{ and } J_1=J'_1 \text{ and } I_2 \subsetneq I'_2) \\
    \text{ or }\ldots
\end{cases}\]
We use in this proof the following notations:
for an element $x \in V$, we denote $m(x)$ (respectively $m'(x)$) the index $m$ (resp $m'$)
such that $x \in I_m$ (respectively $x \in I'_{m'}$).
The same notation will be used for $y \in W$,
except that $\II$ and $\II'$ should be replaced by $\JJ$ and $\JJ'$.
Besides, as in the proof of \cref{lem:Reversing_Edges}, we denote 
\begin{align*}
    V_m&=V  \setminus (I_1 \cup \dots \cup I_{m}); \\
    W_m&=W  \setminus (J_1 \cup \dots \cup J_{m}).
\end{align*}
Analogous notations will be used for $\II'$ and $\JJ'$.
We will prove that if $\NN_{(\II',\JJ')}$ appears 
in the expansion~\eqref{eq:NDec_GamB}
of $\GammaNC(B_{(\II,\JJ)})$, then $(\II,\JJ) \preceq (\II',\JJ')$.

Assume that $I_m=I'_m$ and $J_m=J'_m$ for all $m$ smaller than an integer $m_0 \ge 1$.
We shall prove that $I_{m_0} \subseteq I'_{m_0}$.
Assume $I_{m_0} \neq \emptyset$.
\begin{itemize}
    \item Either $J'_{m_0}$ is empty,
        which forces $I'_{m_0} = V_{m_0-1}'$
        (in particular, $(\II',\JJ')$ has the semi-length $m_0$).
        But as $I_m=I'_m$ for $m <m_0$, we have
        $V_{m_0-1}=V_{m_0-1}'$, so $I_{m_0} \subseteq I'_{m_0}$.
    \item Or $J'_{m_0}$ contains an element $y_0$. 
        As $J_m=J'_m$ for $m<m_0$, one has $W_{m_0-1}=W_{m_0-1}'$.
        Therefore $y_0$ belongs to $W_{m_0-1}$ and for any $x \in I_{m_0}$
        the pair $(x,y_0)$ is an edge of $B_{(\II,\JJ)}$,
        thus, from \cref{prop:Mult_Free_Sum}, one has $m'(x) \le m'(y_0)=m_0$.
        But elements $x$ in $I_{m_0}$ cannot belong to any of the $I'_m=I_m$
        with $m<m_0$,
        therefore we have $x \in I'_{m_0}$.
        We have proved that $I_{m_0} \subseteq I'_{m_0}$,
        which is what we wanted.
\end{itemize}

Fix a positive integer $m_0$ as before and 
assume that $I_m=I'_m$ and $J_m=J'_m$ for $m<m_0$ and $I_{m_0}=I'_{m_0}$.
We shall prove that $J_{m_0} \supseteq J'_{m_0}$.
Again, we consider two cases.
\begin{itemize}
    \item Either $I_{m_0+1}$ is not defined
        (because $(\II,\JJ)$ has semi-length $m_0$), which means that 
        $J_{m_0}=W_{m_0-1}$. But, the hypothesis $J_m=J'_m$ for $m<m_0$
        implies $W_{m_0-1}=W_{m_0-1}'$. Moreover, by definition,
        $J'_{m_0} \subseteq W_{m_0-1}'$ so that 
        $J_{m_0} \supseteq J'_{m_0}$.
    \item Or $I_{m_0+1}$ contains an element $x_0$.
        For each $y$ in $W_{m_0}$, the pair $(x_0,y)$ is an edge of $B_{(\II,\JJ)}$
        and thus, from \cref{prop:Mult_Free_Sum}, one has $m'(x_0) \le m'(y)$.
        But $m'(x_0)=m_0+1$.
        This implies $m'(y) \ge m_0+1$, that is $y \in W_{m_0}'$.
        We have proved that $W_{m_0} \subseteq W_{m_0}'$,
        which, together with $W_{m_0-1}=W_{m_0-1}'$,
        implies that $J_{m_0} \supseteq J'_{m_0}$,
        as wanted.
\end{itemize}

Finally, we have proved that, if $\NN_{(\II',\JJ')}$ appears
in the expansion~\eqref{eq:NDec_GamB}
of the function $\GammaNC(B_{(\II,\JJ)})$, then $(\II,\JJ) \preceq (\II',\JJ')$.
Note that, again from \cref{prop:Mult_Free_Sum},
the basis element $\NN_{(\II,\JJ)}$ appears in this expansion with coefficient $1$.
In other terms the matrix of the family  $(\GammaNC(B_{(\II,\JJ)}))$ in the $\Z$-basis
$\NN_{(\II,\JJ)}$ is unitriangular with respect to the order $\preceq$,
which proves that
$(\GammaNC(B_{(\II,\JJ)}))$ is also a $\Z$-basis of $\WQSym$.
\end{proof}

\subsection{A generating family of the quotient}
\label{sec:BIJ_Span_Quotient}

We will now show that $(B_{(\II,\JJ)})$, where $(\II,\JJ)$ runs over all set compositions,
 is a generating family in the quotient $\BG / \Bideal$.
 As explained in \cref{sec:Third_Main_Result}, 
 together with the results of \cref{sec:Zbasis_Res,sec:CIE_Rel},
this implies that the morphism $\BGammaNC:\BG / \Bideal \to \WQSym$ is an isomorphism.

As in the non-restricted setting, the result follows from a combinatorial lemma
(which is surprisingly simpler than in the non-restricted setting).

\begin{lemma}
    Let $B$ be a bipartite graph on  vertex set $[n]$. Then
    \begin{itemize}
        \item either $B= B_{(\II,\JJ)}$ for some set composition $(\II,\JJ)$;
        \item or $B$ can be written as linear combination of graphs with the 
            same vertex set and more edges in $\BG / \Bideal$.
    \end{itemize}
    \label{LemBiggerGraphsBipartite}
\end{lemma}
\begin{proof}
    Let $[n]=V \sqcup W$ the bipartition of the vertices of $B$.
    For $v \in V$, we denote $\NNN(v)$ the subset of $W$
    of vertices linked to $v$.

    First case: let us suppose that for all $v$ and $v'$ in $V$, we have
    either $\NNN(v) \subseteq \NNN(v')$ or $\NNN(v') \subseteq \NNN(v)$.
    Then one can label the vertices in $V$ by $\{v_1,\dots,v_s\}$
    such that
    \[\NNN(v_1) \supseteq \NNN(v_2) \dots \supseteq \NNN(v_s).\]
    We group together vertices $v_i$ which have the same neighbourhood $\NNN(v_i)$.
    This gives a set composition $(I_1,\dots,I_r)$ of $V$ such that
    \[\NNN(I_1) \supsetneq \NNN(I_2) \dots \supsetneq \NNN(I_r),\]
    where $\NNN(I_m)$ denotes the common value of $\NNN(v)$ for $v \in I_m$.
    Then we define $J_k= \NNN(I_k) \setminus \NNN(I_{k+1})$ for $k<r$
    (these sets are nonempty by definition)
    and $J_r=\NNN(I_r)$ so that, for all $m \le r$,
    \[ \NNN(I_m) =\bigsqcup_{k \geq m} J_k. \]
    This equation precisely says that $B$ is the graph $B_{(\II,\JJ)}$.

    We consider now the second case: there exist $v,v'$ in $V$ and $w,w'$
    in $W$ such that $(v,w)$ and $(v',w')$ belong to edge-set $E_B$ but
    neither $(v,w')$ nor $(v',w)$.
    Let $B_0$ be the graph obtained from $B$ by adding edges from $v$ to $w'$
    and from $v'$ to $w$ (note that it is still bipartite as a directed graph,
    and hence is acyclic).
    The undirected version of this graph
    contains a cycle $C:v \to w' \to v' \to w \to v$,
    whose corresponding set $C^+$ is $\{(v,w'),\ (v',w)\}$ (with the notations of \cref{sub:DefCIE}).
    Then $B=B_0 \setminus C^+$ is the smallest graph appearing in $\CIE_{B_0,C}$
    and thus, in the quotient,  $\BG / \ideal_B$, the graph $B$ can be written as
    a linear combination of bigger graphs ({\it i.e.} with the same set of
    vertices and more edges).
\end{proof}
Let $B$ be a bipartite directed graph with vertex set $[n]$.
Iterating Lemma~\ref{LemBiggerGraphsBipartite},
one can write $B$ as an integral linear combination of $B_{(\II,\JJ)}$
in the quotient space $\BG / \Bideal$.
So $(B_{(\II,\JJ)})$,
where $(\II,\JJ)$ runs over all set compositions
is a generating family for $\BG / \Bideal$.

%%%%%%%%%%%%%%%%%%%%%
%%%%%%%%%%%%%%%%%%%%
\subsection{Third main result}
\label{sec:Third_Main_Result}

We are now ready to prove the following statement.
\begin{theorem}
The space $\Bideal$, spanned by cyclic-inclusion elements,
is the kernel of the surjective morphism $\BGammaNC$
from $\BG$ to $\WQSym$.
    \label{thm:NonCom_Bip}
\end{theorem}
\begin{proof}
    The proof is completely similar to that of \cref{thm:NonCom_NonRestr}.

    Denote $\BKer$ the kernel of $\BGammaNC$. 
    By \cref{prop:CIE_in_Kernel}, it contains $\Bideal$.
    On the one hand (\cref{sec:BIJ_Span_Quotient}),
    we know that $\BG / \Bideal$
    is spanned by the family $(B_{(\II,\JJ)})$.
    On the other hand (\cref{corol:ZBasis_Restr}),
    the family $\GammaNC\left( B_{(\II,\JJ)} \right)$ is a basis of $\WQSym$,
    which implies in particular that the $(B_{(\II,\JJ)})$ are linearly independent
    in $\BG/\BKer$ and hence in $\BG / \Bideal$.

    Therefore $(B_{(\II,\JJ)})$ is a basis of $\BG / \Bideal$
    and $\BGammaNC$ is an isomorphism from $\BG / \Bideal$ to $\WQSym$ 
    (it sends a basis on a basis), which concludes the proof.
\end{proof}

\subsection{Unlabeled commutative framework and fourth main result}
We will use the following obvious notations for the commutative bipartite framework:
$\ov{\BG}$ is the subspace of $\ov{\G}$ spanned by unlabeled bipartite graph
and $\BGamma$ is the restriction of $\Gamma$ to $\ov{\BG}$.

Moreover, we denote $\ov{\Bideal}$ the space spanned by $\CIE_{\ov{G},\ov{C}}$,
where $\ov{G}$ runs over unlabeled bipartite directed graphs
and $\ov{C}$ over cycles in the undirected version of $\ov{G}$.
Equivalently, $\ov{\Bideal}$ is the image of $\Bideal$ by $\phi_u$.

\begin{theorem}
    The ideal $\ov{\Bideal}$, spanned by inclusion-exclusion elements,
is the kernel of the surjective morphism $\BGamma$
from $\ov{\BG}$ to $\QSym$.
    \label{thm:Com_Bip}
\end{theorem}
\begin{proof}
    The proof is identical to that of \cref{thm:Com_NonRestr},
    using \cref{thm:NonCom_Bip} instead of \cref{thm:NonCom_NonRestr}.
\end{proof}

\section{Application of the main result to Kerov character polynomials}
\label{sect:Kerov}

In this section, we present our application of Theorem \cref{thm:Com_Bip}
to the theory of Kerov character polynomials.
We do not obtain new results, but are able to significantly
simplify some existing proofs.

\subsection{A family of invariant functionals}
We start by defining combinatorially a family of linear functions
$I_\nu : \BG \to \C$ indexed by integer partitions\footnote{As usual, an {\em integer
partition} is a non-increasing list of positive integers.},
whose kernels contain inclusion-exclusion elements.

\begin{definition}
    \begin{itemize}
        \item A {\em decorated bipartite graph} is a pair $(B,h)$
            where $B$ is a graph with vertex set bipartition $V \sqcup W$
            and $h$ a function $V \to \{1,2,\cdots\}$ such that
            \[\sum_{v \in V} h(v) = |W|.\]
        \item A connected decorated bipartite graph is said to be {\em expander}
            if, for any non-empty proper subset $U$ of $V$ (that is $U \neq \emptyset, V$),
            \[ |\NNN(U)| > \sum_{u \in U} h(u), \]
            where $\NNN(U)$ is the neighbourhood of $U$,
            {\em i.e.} the set of vertices of $W$ having at least one neighbour in $U$.
        \item A decorated bipartite graph is said to be {\em expander}
            if all its connected components are
            (in particular, if $V \sqcup W$ is the vertex set
            of a connected component, then $\sum_{v \in V} h(v)=|W|$).
        \item The {\em type} of a decorated bipartite graph is the integer partition
            obtained by sorting the multiset $h(V)$ in non-increasing order.
    \end{itemize}
\end{definition}

\begin{example}
    Consider the bipartite graph $B$ of \cref{fig:ex_expander} (without the dashed edge)
    and let $h$ be given by $h(5)=1$, $h(3)=2$ and $h(8)=3$.
    Then $(B,h)$ is a decorated bipartite graph of type $(3,2,1)$.
    It is {\em not} expander as the neighbourhood of $\{3,5\}$ has size $3$ while $h(3)+h(5)=3$
    (notice the strong inequality in the definition of expander).
    If we add the dashed edge, we get an expander graph.
\end{example}
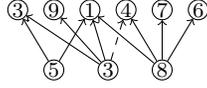
\begin{figure}
    \[  \begin{tikzpicture}[font=\scriptsize,scale=.32]
        \tikzstyle{vertex}=[circle,draw=black,inner sep=.3pt];
        \node[vertex] (v5) at (0,0) {$5$};
        \node[vertex] (v3) at (2.25,0) {$3$};
        \node[vertex] (v8) at (4.5,0) {$8$};
        \node[vertex] (v2) at (-1.5,2.5) {$3$};
        \node[vertex] (v9) at (0,2.5) {$9$};
        \node[vertex] (v1) at (1.5,2.5) {$1$};
        \node[vertex] (v4) at (3,2.5) {$4$};
        \node[vertex] (v7) at (4.5,2.5) {$7$};
        \node[vertex] (v6) at (6,2.5) {$6$};
        \draw[->] (v5) -- (v2);
        \draw[->] (v5) -- (v1);
        \draw[->] (v3) -- (v2);
        \draw[->] (v3) -- (v9);
        \draw[->] (v3) -- (v1);
        \draw[dashed,->] (v3) -- (v4);
        \draw[->] (v8) -- (v1);
        \draw[->] (v8) -- (v4);
        \draw[->] (v8) -- (v7);
        \draw[->] (v8) -- (v6);
    \end{tikzpicture} \]
    \caption{Example of non-expander (without the dashed edge) and expander (with the dashed edge) graphs.}
    \label{fig:ex_expander}
\end{figure}
\begin{remark}
    There are many variants of the definition of {\em expander graphs} in the literature.
    The one given here is a generalization of having {\em left-vertex expansion ratio}
    at least $h$ (for a given integer $h$), see \cite[Definition 12.7]{Expanders}.
    Expander graphs have found a lot of applications in analysis of communication networks,
    in the theory of error correcting codes and in the theory of pseudorandomness:
    we refer to \cite{Expanders} for a survey article.
    However, the way they appear here seems very different to what is usually done in the
    literature.
\end{remark}
Expander graphs are known to encode some kind of strong connectivity of the graphs.
In particular, trees (here, a tree is connected graph whose undirected version does not contain cycles)
are not expanders (except for trivial cases),
which is stated in the following lemma.

\begin{lemma}
    Let $B$ be tree with vertex set bipartition $V \sqcup W$ and $h:V \to \{1,2,\cdots\}$.
    Then $(B,h)$ is expander if and only if every connected component of $B$ contains exactly one
    vertex in $V$ and $h$ associates to each vertex in $V$ its number of neighbours.
    \label{LemForestsExpanders}
\end{lemma}
\begin{proof}
    It is enough to prove that $(B,h)$ can not be expander unless 
    $B$ has one vertex of $V$ per connected component or, equivalently,
    unless all vertices in $W$ have degree $1$.
    The remaining part of the lemma then follows easily.

    Let us do a proof by contradiction and 
    assume there is a vertex $w$ of $W$ of degree at least $2$. 
    Without loss of generality, we may assume that $B$ is connected.
    As $B$ is a tree, if we remove $w$, the graph obtained from $B$ has several connected components:
    denote $V_1$, \ldots, $V_r$ the intersections of $V$ with these connected components ($r \ge 2$).
    
    The union of the neighbourhoods $\NNN(V_1)$, \ldots, $\NNN(V_r)$ is clearly $W$,
    while two sets in this list have only $w$ in common,
    so that
    \[\sum_{i=1}^r |\NNN(V_i)| = |W| + (r-1). \]
    But, by hypothesis,
    \[ \sum_{i=1}^r \left( \sum_{v \in V_i} h(v) \right)=\sum_{v \in V} h(v)=|W|\]
    which is incompatible with the strict inequalities
    ($V_1$, \ldots, $V_r$ are non-empty by definition and proper subsets of $V$ because $r \ge 2$):
    \[\text{for every $i$ in $\{1,\cdots,r\}$, }\ |\NNN(V_i)|>\sum_{v \in V_i} h(v).\qedhere\]
\end{proof}

We can now define the functions $I_\nu$.

\begin{definition}
    Let $\nu$ be an integer partition and $B$ a bipartite graphs with $\c$ connected components.
    Then $(-1)^\c I_\nu(B)$ is, by definition, the number of functions $h: V \to \{1,2,\cdots\}$
    such that $(B,h)$ is an expander decorated bipartite graph of type $\nu$.

    The function $I_\nu$ is then extended by linearity to the bipartite graph algebra $\BG$.
\end{definition}

\begin{proposition}
    For any bipartite graph $B$ and cycle $C$ of $B$, one has:
    \[I_\nu(\CIE_{B,C})=0.\]
    \label{prop:Iinv}
\end{proposition}
\begin{proof}
    See \cite[Lemma 8.3]{DolegaFeraySniady2008}.
\end{proof}
\begin{remark}
    While all elements in the statement of \cref{prop:Iinv} are combinatorial,
    the proof given in \cite{DolegaFeraySniady2008} involves computations of Euler characteristic.
    An {\em elementary} proof would certainly be interesting.
\end{remark}

\subsection{Background on Kerov character polynomials}
We only present here what is strictly necessary to explain our
application of \cref{thm:Com_Bip}. 
As this is not central in the paper,
we assume some familiarity of the reader with representation theory of symmetric groups.
Details and motivations can be found in \cite{DolegaFeraySniady2008} and references therein.

Let $\mu$ be fixed integer partition. Consider the function
\[\Ch_{\mu}(\lambda)=
\begin{cases}
    |\lambda|(|\lambda|-1)\cdots (|\lambda|-|\mu|+1) \frac{\chi^\lambda_{\mu\, 1^{|\lambda|-|\mu|}}}
    {\dim(\lambda)} &\text{ if }|\lambda| \ge |\mu| ;\\
    0 &\text{ if }|\lambda| < |\mu|.
\end{cases}\]
Here $\lambda$ is a Young diagram, $\dim(\lambda)$ the dimension of the associated irreducible representation
of the symmetric group and $\chi^\lambda_{\mu\, 1^{|\lambda|-|\mu|}}$ the associated character
evaluated on a permutation of cycle-type $\mu \cup (1^{|\lambda|-|\mu|})$.

Consider a diagram given by its modified multirectangular coordinates $(p_1,\cdots,p_m)$
and $(q_1,\cdots,q_m)$,
that is
$$\lambda(\s{p},\s{q}):=\underbrace{\sum_{i \geq 1} q_i, \ldots, \sum_{i \geq 1} q_i}_{p_1 \text{ times}},\underbrace{\sum_{i \geq 2} q_i,\ldots,\sum_{i \geq 2} q_i}_{p_2 \text{ times}},\ldots$$
It has been shown (see {\em e.g.} \cite[Theorem 1.5.1]{F'eray2008}) that
\begin{equation}
    \Ch_{\mu}(\lambda(\s{p},\s{q}))=
\sum_{\sigma,\tau \in S_k \atop \sigma \, \tau = \pi} (-1)^{\kappa(\tau) +r} \Delta(B(\sigma,\tau))(\s{p},\s{q}), 
\label{eq:StanleyForm}
\end{equation}
where:
\begin{itemize}
    \item $k$ and $r$ are respectively the size and the length of $\mu$ and $S_k$ the symmetric group of size $k$;
    \item $\pi$ is a fixed (arbitrary) permutation of cycle-type $\mu$;
    \item $\kappa(\tau)$ is the number of cycles of $\tau$;
    \item $B(\sigma,\tau)$ is a bipartite graph associated to the pair of permutations $\sigma$ and $\tau$
        (its precise definition is not important here);
    \item $\Delta(B)$ is a two-alphabet version of $\Gamma(B)$, namely:
        \[\Delta(B)(\s{p},\s{q})= \sum_{\gf{f: V \sqcup W \to \N}{f\ B \text{ non-decreasing}}} \left(
        \prod_{v \in V} p_{f(v)} \cdot \prod_{w \in W} q_{f(w)}  \right),\]
        where $V \sqcup W$ is the proper bipartition of vertices of $B$ without isolated vertices in $W$.
\end{itemize}

Another family of functions of interest is the family of {\em free cumulants},
which can be defined as follows:
\begin{equation}
    R_{k+1}(\lambda(\s{p},\s{q}))=
    \sum_{{\sigma,\tau \in S_k \atop \sigma\, \tau = (1\ 2\ \dots\ k)} \atop \kappa(\sigma)+\kappa(\tau)=k+1}
    (-1)^{\kappa(\tau) +1} \Delta(B(\sigma,\tau))(\s{p},\s{q}).
\label{eq:StanleyFormR}
\end{equation}
The restriction $\kappa(\sigma)+\kappa(\tau)=k+1$ imposed in the summation index is in fact 
equivalent (under the assumption $\sigma \, \tau = (1\ 2\ \dots\ k)$) to the fact that
$B(\sigma,\tau)$ has no cycles.

In 2001, S. Kerov proved that, for each partition $\mu$, there exists a polynomial $K_\mu$,
now called Kerov polynomial, such that, for every Young diagram $\lambda$, one has
\begin{equation}
    \Ch_\mu(\lambda) = K_\mu(R_2(\lambda),R_3(\lambda),\dots,R_{|\mu|+1}(\lambda)).
    \label{eq:Kerov}
\end{equation}
He then conjectured -- see \cite{Biane2003} -- that $K_{(k)}$ has non-negative
coefficients for any positive integer $k$.
This result was proved by the author in \cite{F'eray2008} and an explicit combinatorial
interpretation of the coefficients was given in \cite{DolegaFeraySniady2008}.
We explain in next Section how \cref{thm:Com_Bip} and the invariants $I_\nu$
may be used to simplify the arguments in these papers.

\subsection{Application of our main result}
Similarly to $\Gamma$, the two-alphabet version $\Delta$ can be extended
by linearity to the bipartite graph algebra $\G_b$.
Consider elements $G_{\Ch_\mu}$ and $G_{R_k}$ in the graph algebra such that
\[\Ch_\mu(\lambda(\s{p},\s{q})) = \Delta({G_{\Ch_\mu}})( \s{p},\s{q}), \quad R_k(\lambda(\s{p},\s{q}))=\Delta({G_{R_k}})( \s{p},\s{q}),\]
that is
\begin{align*}
    G_{\Ch_\mu} & = \sum_{\sigma,\tau \in S_k \atop \sigma \, \tau = \pi} (-1)^{\kappa(\tau) +r} B(\sigma,\tau)  ;\\
    G_{R_{k+1}} & = \sum_{ {\sigma,\tau \in S_k \atop \sigma\, \tau = (1\ 2\ \dots\ k)} \atop \kappa(\sigma)+\kappa(\tau)=k+1}
        (-1)^{\kappa(\tau) +1} B(\sigma,\tau).
\end{align*}
Then observe that the \cref{eq:Kerov} for any Young diagram $\lambda$
implies that
\[   \Ch_\mu(\lambda(\s{p},\s{q})) = K_\mu(R_2(\lambda(\s{p},\s{q})),R_3(\lambda(\s{p},\s{q})),\dots,R_{|\mu|+1}(\lambda(\s{p},\s{q}))) \]
as polynomials in infinitely many variables $p_1,q_1,p_2,q_2,\cdots$, so that
\[\Delta(G_{\Ch_\mu})= K_\mu\big(\Delta(G_{R_2}),\dots,\Delta(G_{R_k}) \big) 
=\Delta\big(K_\mu(G_{R_2},\dots,G_{R_k})\big).\]
Recall indeed that the product in the graph algebra is given by disjoint union of graphs
and that $\Delta$ is clearly an algebra morphism with respect to this product.

But sending $p_i,q_i \to x_i$ sends $\Delta(B)$ to $\Gamma(B)$,
thus the difference
\[A:= G_{\Ch_\mu} - K_\mu(G_{R_2},\dots,G_{R_k}) \]
lies in $\Ker(\Gamma)$.
By \cref{thm:Com_Bip}, it lies in $\Bideal$ and thus \cref{prop:Iinv}
implies that $I_\nu(A)=0$ for any partition $\nu$.

But, one can easily seen from \cref{LemForestsExpanders} 
(recall that graphs appearing in $G_{R_k}$ have no cycles) that 
\[I_\nu(G_{R_{i_1}} \cdots G_{R_{i_\ell}}) =\begin{cases}
    (-1)^{\ell} &\begin{tabular}{l} if $\nu$ is obtained by antisorting\\ $(i_1-1,\dots,i_\ell-1)$ in decreasing order;
    \end{tabular}\\
    0 &\text{ otherwise.}
\end{cases}\]
Therefore $I_\nu(K_\mu(G_{R_2},\dots,G_{R_k}))$ is, up to a sign, the coefficient of the monomial
\[ \prod_{i=1}^{\ell(\nu)} R_{\nu_i+1} \]
in $K_\mu$. From the relation $I_\nu(A)=0$, we get that it is also equal to $I_\nu(G_{\Ch_\mu})$.
This last quantity is a signed enumeration of expander graphs,
so that we obtain a signed combinatorial interpretation for coefficients of Kerov polynomials.

This signed combinatorial interpretation is equivalent to \cite[Theorem 1.6]{DolegaFeraySniady2008}.
In the case $\mu=(k)$, the signs disappear and the non-negativity of the coefficients of $K_{(k)}$ follows.

\subsection{Comparison with the proofs given in \cite{DolegaFeraySniady2008}}
In \cite{DolegaFeraySniady2008}, two proofs of the result above were given.
The first one is quite different from the one sketched above.
The second one also used cyclic inclusion-exclusion and \cref{prop:Iinv},
but a huge part of the proof was dedicated to proving the fact that the quantity $A$ belongs to $\Bideal$
-- see \cite[Sections 3 and 4]{F'eray2008}.
With \cref{thm:Com_Bip}, it follows immediately from the fact that $\Delta(A)=0$.

Besides, the proof that $A \in \Bideal$ given in \cite[Sections 3 and 4]{F'eray2008}, uses the structure of the symmetric group,
while the argument that we use here works if we replace $\Ch_\mu$ by any function that
has an expression similar to \cref{eq:StanleyForm} -- for instance the zonal characters studied in \cite{FeraySniady2011}.
Note that the first proof of paper \cite{DolegaFeraySniady2008}
also extends readily to zonal characters, so the result that we obtain that way is not new.

\bibliographystyle{abbrv}
\bibliography{../courant.bib}

\end{document}